\documentclass[11pt]{article}
\usepackage[margin=1in]{geometry} 
\usepackage{amssymb,amsfonts,amsmath,bbm,mathrsfs,stmaryrd,mathtools}
\usepackage{xcolor}
\usepackage{url}

\usepackage{enumerate}
\usepackage{centernot}

\usepackage{hyperref}
\hypersetup{colorlinks,
             linkcolor=black!75!red,
             citecolor=blue,
             pdftitle={},
             pdfproducer={pdfLaTeX},
             pdfpagemode=None,
             bookmarksopen=true
             bookmarksnumbered=true}

\usepackage{tikz}
\usetikzlibrary{arrows,arrows.meta,calc,decorations.pathreplacing,decorations.markings,intersections,shapes.geometric,through,fit,shapes.symbols,positioning,decorations.pathmorphing}

\usepackage{braket}

\usepackage[amsmath,thmmarks,hyperref]{ntheorem}
\usepackage{cleveref}

\creflabelformat{enumi}{#2(#1)#3}

\crefname{section}{Section}{Sections}
\crefformat{section}{#2Section~#1#3} 
\Crefformat{section}{#2Section~#1#3} 

\crefname{subsection}{\S}{\S\S}
\AtBeginDocument{%
  \crefformat{subsection}{#2\S#1#3}%
  \Crefformat{subsection}{#2\S#1#3}%
}

\crefname{subsubsection}{\S}{\S\S}
\AtBeginDocument{%
  \crefformat{subsubsection}{#2\S#1#3}%
  \Crefformat{subsubsection}{#2\S#1#3}%
}

\theoremstyle{plain}

\newtheorem{lemma}{Lemma}[section]
\newtheorem{proposition}[lemma]{Proposition}
\newtheorem{corollary}[lemma]{Corollary}
\newtheorem{theorem}[lemma]{Theorem}

\theoremstyle{nonumberplain}
\newtheorem{theoremN}{Theorem}
\newtheorem{propositionN}{Proposition}

\theoremstyle{plain}
\theorembodyfont{\upshape}
\theoremsymbol{\ensuremath{\blacklozenge}}

\newtheorem{definition}[lemma]{Definition}
\newtheorem{example}[lemma]{Example}
\newtheorem{examples}[lemma]{Examples}
\newtheorem{remark}[lemma]{Remark}
\newtheorem{remarks}[lemma]{Remarks}
\newtheorem{convention}[lemma]{Convention}

\crefname{definition}{definition}{definitions}
\crefformat{definition}{#2definition~#1#3} 
\Crefformat{definition}{#2Definition~#1#3} 

\crefname{ex}{example}{examples}
\crefformat{example}{#2example~#1#3} 
\Crefformat{example}{#2Example~#1#3} 

\crefname{remark}{remark}{remarks}
\crefformat{remark}{#2remark~#1#3} 
\Crefformat{remark}{#2Remark~#1#3} 

\crefname{remarks}{remark}{remarks}
\crefformat{remarks}{#2remark~#1#3} 
\Crefformat{remarks}{#2Remark~#1#3} 

\crefname{convention}{convention}{conventions}
\crefformat{convention}{#2convention~#1#3} 
\Crefformat{convention}{#2Convention~#1#3} 

\crefname{notation}{notation}{notations}
\crefformat{notation}{#2notation~#1#3} 
\Crefformat{notation}{#2Notation~#1#3} 

\crefname{table}{table}{tables}
\crefformat{table}{#2table~#1#3} 
\Crefformat{table}{#2Table~#1#3}

\crefname{lemma}{lemma}{lemmas}
\crefformat{lemma}{#2lemma~#1#3} 
\Crefformat{lemma}{#2Lemma~#1#3} 

\crefname{proposition}{proposition}{propositions}
\crefformat{proposition}{#2proposition~#1#3} 
\Crefformat{proposition}{#2Proposition~#1#3} 

\crefname{corollary}{corollary}{corollaries}
\crefformat{corollary}{#2corollary~#1#3} 
\Crefformat{corollary}{#2Corollary~#1#3} 

\crefname{theorem}{theorem}{theorems}
\crefformat{theorem}{#2theorem~#1#3} 
\Crefformat{theorem}{#2Theorem~#1#3} 

\crefname{enumi}{}{}
\crefformat{enumi}{(#2#1#3)}
\Crefformat{enumi}{(#2#1#3)}

\crefname{assumption}{assumption}{Assumptions}
\crefformat{assumption}{#2assumption~#1#3} 
\Crefformat{assumption}{#2Assumption~#1#3} 

\crefname{equation}{}{}
\crefformat{equation}{(#2#1#3)} 
\Crefformat{equation}{(#2#1#3)}

\numberwithin{equation}{section}

\theoremstyle{nonumberplain}
\theoremsymbol{\ensuremath{\blacksquare}}

\newtheorem{proof}{Proof}
\newcommand\pf[1]{\newtheorem{#1}{Proof of \Cref{#1}}}

\newcommand\bC{{\mathbb C}}

\newcommand\bR{{\mathbb R}}
\newcommand\bS{{\mathbb S}}

\newcommand\bZ{{\mathbb Z}}

\newcommand\cC{{\mathcal C}}
\newcommand\cD{{\mathcal D}}

\newcommand\cF{{\mathcal F}}
\newcommand\cG{{\mathcal G}}
\newcommand\cH{{\mathcal H}}

\newcommand\cM{{\mathcal M}}

\newcommand\cS{{\mathcal S}}

\newcommand\cV{{\mathcal V}}

\DeclareMathOperator{\id}{id}

\newcommand{\cat}[1]{\textsc{#1}}

\newcommand{\qedhere}{\mbox{}\hfill\ensuremath{\blacksquare}}

\title{Metric enrichment, finite generation, and the path comonad}
\author{Alexandru Chirvasitu}

\begin{document}

\date{}

\newcommand{\Addresses}{{
  \bigskip
  \footnotesize

  \textsc{Department of Mathematics, University at Buffalo, Buffalo,
    NY 14260-2900, USA}\par\nopagebreak \textit{E-mail address}:
  \texttt{achirvas@buffalo.edu}

}}

\maketitle

\begin{abstract}
  We prove a number of results involving categories enriched over \textsc{CMet}, the category of complete metric spaces with possibly infinite distances. The category \textsc{CPMet} of intrinsic complete metric spaces is locally $\aleph_1$-presentable, closed monoidal, and comonadic over \textsc{CMet}. We also prove that the category \textsc{CCMet} of convex complete metric spaces is not closed monoidal and characterize the isometry-$\aleph_0$-generated objects in \textsc{CMet}, \textsc{CPMet} and \textsc{CCMet}, answering questions by Di Liberti and Rosick\'{y}. Other results include the automatic completeness of a colimit of bi-Lipschitz morphisms of complete metric spaces and a characterization of those pairs (metric space, unital $C^*$-algebra) that have a tensor product in the \textsc{CMet}-enriched category of unital $C^*$-algebras.
\end{abstract}

\noindent {\em Key words: complete metric space; intrinsic metric; gluing; convex; monoidal closed; enriched; tensored; locally presentable; locally generated; colimit; internal hom}

\vspace{.5cm}

\noindent{MSC 2020: 54E50; 54E40; 51F30; 18A30; 18C35; 18D20; 18D15; 46L05; 46L09}

\tableofcontents

\section*{Introduction}

Denote by \cat{CMet} the category of complete metric spaces with non-expansive maps as morphisms, and distance functions allowed infinite values (see \Cref{cv:met}). \cite[Example 2.3 (2)]{ar-ap} notes that \cat{CMet} is {\it symmetric monoidal closed} \cite[\S\S 1.1, 1.4, 1.5]{kly}, so it is a good candidate category for {\it enriching} over in the sense of \cite{kly}.

Many categories of interest in functional analysis are {\it \cat{CMet}-enriched} or {\it \cat{CMet}-categories} in the sense of \cite[\S 1.2]{kly}: for every two objects $x,y\in \cC$ in the category of interest there is a {\it morphism object} $[x,y]\in \cat{CMet}$, there is an associative composition
\begin{equation*}
  [y,z]\otimes [x,y]\to [x,z]
\end{equation*}
for an appropriate monoidal structure on $\cat{CMet}$, etc. Natural examples are in rich supply:
\begin{itemize}
\item \cat{CMet} is self-enriched, the space of non-expansive maps between two complete metric spaces being metrized with the supremum distance;
\item \cat{Ban}, consisting of Banach spaces and linear maps of norm $\le 1$;
\item the category $\cat{BanAlg}_1$ of (complex) unital Banach algebras or its variations $\cat{BanAlg}^*_1$ (unital complex Banach $*$-algebras), $\cat{BanAlg}_{c,1}$ ({\it commutative} unital Banach algebras), etc.;
\item $\cC^*_1$, the category of unital $C^*$-algebras, or $\cC^*_{c,1}$, that of {\it commutative} unital $C^*$-algebras. 
\end{itemize}

Such metric-flavored category-theoretic considerations are by now pervasive in the literature: in discussing universal (Gurarii) Banach spaces \cite{kub-enr,lup-fr}, or universal operators thereon \cite{gk-univ}, or more general issues of approximate embeddability \cite{rt,ar-ap}; these are only a handful of examples, each with its own extensive cited literature.

The initial motivation for the present paper were a number of questions arising naturally in \cite{dr}, in studying {\it local generation} in this enriched setting. Roughly speaking, an object $x$ in a category $\cC$ is $\kappa$-generated for a cardinal $\kappa$ if $\mathrm{hom}_{\cC}(x,-)$ preserves ``sufficiently directed'' colimits; when the category is $\cV$-enriched one can instead consider
\begin{equation*}
  [x,-]:\cC\to \cV,
\end{equation*}
leading to the notion studied in loc.cit. Formally, aggregating, say, \cite[Definition 1.13]{ar} and \cite[Definitions 2.1 and 4.1]{dr}:

\begin{definition}\label{def:gen}
  Let $\kappa$ be a regular cardinal.
  \begin{itemize}
  \item A poset $(I,\le)$ is {\it $\kappa$-directed} if every subset of $I$ of cardinality $<\kappa$ as an upper bound.
  \item A {\it $\kappa$-directed colimit} in a category is a colimit of a functor defined on a $\kappa$-directed poset (regarded as a category, with an arrow $i\to j$ when $i\le j$).
  \item An object $x\in \cC$ in a category is {\it $\cM$-$\kappa$-generated} for a class of morphisms $\cM$ if
    \begin{equation*}
      \mathrm{hom}(x,-):\cC\to \cat{Set}
    \end{equation*}
    preserves $\kappa$-directed colimits of morphisms in $\cM$.
  \item Similarly, if $\cC$ is $\cV$-enriched, $x$ is {\it $\cM$-$\kappa$-generated in the enriched sense} (or {\it enriched $\cM$-$\kappa$-generated}) if the above colimit-preservation condition holds for the enriched-hom functor
    \begin{equation*}
      [x,-]:\cC\to \cV
    \end{equation*}
    instead.
  \end{itemize}
\end{definition}
Being $\kappa$-generated is a kind of smallness condition: in, say, categories of modules over rings, it literally means being generated by fewer than $\kappa$ elements \cite[Proposition 3.10]{ar}. For that reason, it is also customary to refer to $\aleph_0$-generated objects as {\it finitely generated}; this is the finite generation of the paper's title.

\cite[Remark 6.9]{dr} briefly considers \cat{CCMet} as another candidate to enrich over: this is the category of complete {\it convex} metric spaces, i.e. those for which pairs of points a finite distance apart can be connected by curves that realize that distance (this differs slightly from the definition adopted in \cite{dr}; see \Cref{def:cvx} and surrounding discussion).

Given that the finite segments $[0,\ell]\in \cat{CCMet}$ are in a sense the basic building blocks of \cat{CCMet}, it is natural to ask, as \cite[Remark 6.9]{dr} does, whether they are enriched-finitely-generated in the sense of \Cref{def:gen}, with respect to the class of isometries. It turns out that not only is the answer negative, but finite generation is rather difficult to come by in any of the categories of interest. Summarizing \Cref{th:fgen,th:fgenp} and \Cref{cor:fgenc}:

\begin{theoremN}
  In any of the categories
  \begin{itemize}
  \item \cat{CMet} of complete metric spaces;
  \item \cat{CPMet} of complete {\it path} metric spaces;
  \item or \cat{CCMet} of complete convex metric spaces
  \end{itemize}
  the isometry-$\aleph_0$-generated objects are precisely the finite discrete metric spaces, i.e. those with all pairwise distances infinite.  \qedhere
\end{theoremN}

This also generalizes \cite[Proposition 5.19]{ar-ap}, which proves that in \cat{CMet}, the only isometry-$\aleph_0$-generated finite spaces are the discrete ones (i.e. in that statement finiteness is assumed).

{\it Path} or {\it intrinsic} metric spaces are recalled in \Cref{def:lspace}: they are those for which points a finite distance $\ell$ apart are connectable with curves of length arbitrarily close to $\ell$; they thus intermediate between plain (complete) metric spaces and convex ones.

The appearance of \cat{CPMet} in the discussion is at least in part motivated by another question asked in \cite{dr} (immediately preceding \cite[Remark 6.10]{dr}): whether \cat{CCMet} is monoidal closed. It is not (\Cref{ex:ccmetnclosed}), but essentially because the right adjoint to the inclusion functor
\begin{equation*}
  \iota:\cat{CPMet}\subset \cat{CMet}
\end{equation*}
fails, in general, to produce convex spaces: see \Cref{pr:whencorefl} and \Cref{cor:ihomchar}. \cat{CPMet}, on the other hand, is much better behaved; coalescing \Cref{le:cpmetmon}, \Cref{cor:cpclosed}, and \Cref{th:cpcomon,th:cppres}:

\begin{theoremN}
  The full subcategory
  \begin{equation*}
    \cat{CPMet}\subset \cat{CMet}
  \end{equation*}
  of complete path metric spaces is
  \begin{itemize}
  \item locally $\aleph_1$-presentable (so in particular complete and cocomplete);
  \item and closed monoidal. 
  \end{itemize}
  Furthermore, the inclusion functor is comonadic (\Cref{def:comon}), so in particular a left adjoint.  \qedhere
\end{theoremN}

We highlight a number of pathologies in otherwise well-behaved metric-enriched categories:
\begin{itemize}
\item the failure of \cat{CCMet} to be monoidal closed in \Cref{ex:ccmetnclosed};
\item the paucity of $\aleph_0$-generated objects in \cat{CPMet} or \cat{CCMet} (\Cref{th:fgenp} and \Cref{cor:fgenc}).
\end{itemize}
All of this requires piecing together metric spaces by the {\it gluing} process of \cite[\S 3.1.2]{bbi} (see \Cref{subse:glue} below). Gluing, say, metric spaces $X_i$, $i=1,2$ along a common subspace is nothing but a pushout in the category \cat{Met} of (perhaps incomplete) metric spaces, and our examples need such colimits to have various desired properties (completeness, convexity, etc.). This is ensured by a number of auxiliary results I have not been able to locate in the literature.

To state a joint summary of \Cref{th:largepush,th:tree}, recall (e.g. \cite[Definition 1.4.6]{bbi}) that a map
\begin{equation*}
  f:(X,d_X)\to (Y,d_Y)
\end{equation*}
in \cat{Met} is {\it bi-Lipschitz} if there are both bounds to how much it can scale distances, either up or down: for some $C,C'>0$ we have 
\begin{equation*}
  C d_X(x,x')\le d_Y(fx,fx')\le C' d_X(x,x'),\ \forall x,x'\in X.
\end{equation*}
For added precision, we incorporate the constants into the term and call such maps {\it $(C,C')$-bi-Lipschitz}. The two aforementioned theorems then amalgamate to

\begin{theoremN}
  Let $\Gamma$ be an oriented forest (in the graph-theoretic sense) of finite diameter $D$ and
  \begin{equation*}
    F:\Gamma\to \cat{CMet}
  \end{equation*}
  a functor consisting of $(C,1)$-bi-Lipschitz morphisms.
  \begin{enumerate}[(a)]
  \item The colimit $(X,d):=\varinjlim F$ of $F$ in \cat{Met} is then automatically complete, and hence also a colimit in \cat{CMet}.
  \item And the canonical morphisms
    \begin{equation*}
      F(v)\to X,\ \text{$v$ a vertex of $\Gamma$}
    \end{equation*}
    are $(C',1)$-bi-Lipschitz with $C'$ depending only on $C$ and the diameter $D$.  \qedhere
  \end{enumerate}
\end{theoremN}

Gluing is also helpful in rendering a metric space convex. This produces not quite a reflection of \cat{CMet} into \cat{CCMet}, but rather a {\it weak} reflection (it will not, in general, have the requisite universality property requisite of a reflection functor). Nevertheless, \Cref{pr:cvxcmpl} reads

\begin{propositionN}
  For any complete metric space $(X,d)\in \cat{CMet}$, attaching intervals of length $d(x,x')<\infty$ with endpoints $x,x'\in X$ for any point pair not already connected by such an interval produces a complete convex metric space.  \qedhere
\end{propositionN}

As somewhat of a side-note, but in the same general circle of ideas, we identify in \Cref{se:aside} those pairs
\begin{equation*}
  X\in \cat{CMet},\quad C\in \cC^*_1 :=\text{unital $C^*$-algebras}
\end{equation*}
that have a {\it tensor product} $X\otimes C$. This is by definition a unital $C^*$-algebra that represents the functor
\begin{equation*}
  [X,[C,-]]:\cC^*_1\to \cat{CMet},
\end{equation*}
and whether or not such tensor products always exist in an enriched category is yet another measure of how convenient it is to work with ($\cV$-enriched categories admitting tensor products in this sense are called {\it $\cV$-tensored} \cite[\S 3.7]{kly}). In the context of metric enrichment, there is a discussion of the matter in \cite[\S 4]{ar-ap}.

The earlier \cite[Proposition 3.11]{ck} says that the category $\cC^*_{c,1}$ of {\it commutative} unital $C^*$-algebras is \cat{CMet}-tensored. On the other hand, \Cref{th:notensfull} below negates the existence of tensors in $\cC^*_1$ fairly strongly: in a sense, only the ``obvious'' ones exist.

\begin{theoremN}
  For a complete metric space $X\in\cat{CMet}$ and a unital $C^*$-algebra $C\in \cC^*_1$ the tensor product $X\otimes C\in \cC^*_1$ exists if and only if one of the following conditions holds:
  \begin{itemize}
  \item $X$ has cardinality $\le 1$;
  \item or $C$ has dimension $\le 1$.  \qedhere
  \end{itemize}  
\end{theoremN}

\subsection*{Acknowledgements}

I am grateful for J. Rosick\'{y}'s insightful comments. 

This work is partially supported by NSF grant DMS-2001128.

\section{Preliminaries}\label{se:prel}

\cite[Remark 6.9]{dr} makes a number of observations on the category \cat{CCMet} of {\it convex} complete generalized metric spaces, where
\begin{itemize}
\item `generalized' means that distances are allowed infinite values;
\item and convexity for a metric space $(X,d)$ is as in, say, \cite[\S 2.5]{kk}: for every $x\ne y\in X$ there is some $z\ne x,y$ {\it metrically between} $x$ and $y$ in the sense that
  \begin{equation*}
    d(x,y) = d(x,z) + d(z,y). 
  \end{equation*}
\end{itemize}

\begin{convention}\label{cv:met}
  It is very natural, in the context of the present discussion, to work with possibly-infinite metrics; for that reason, we adopt the terminology of \cite[Defiition 1.1.1]{bbi}: the phrase `metric space' allows for infinite distances. If, on occasion, we encounter $\bR_{\ge 0}$-valued metrics and wish to emphasize the matter, we refer to these as {\it finite} distance functions or metrics.
\end{convention}

Keeping this possible distance infinitude in mind, we write
\begin{itemize}
\item \cat{Met} for the category of metric spaces;
\item and \cat{CMet} for that of {\it complete} metric spaces (following, say, \cite[Example 2.3 (2)]{ar-ap} and \cite[\S 6]{dr}).
\end{itemize}
In both cases the morphisms are the {\it non-expansive} maps (or the {\it contractions}) $f:(X,d_X)\to (Y,d_Y)$:
\begin{equation}\label{eq:1lip}
  d_Y(fx,fx')\le d_X(x,x'),\ \forall x,x'\in X.
\end{equation}
We also refer to contractions as {\it 1-Lipschitz} maps, per \cite[Definition 1.1]{grm}: $\lambda$-Lipschitz, for positive $\lambda$, would mean \Cref{eq:1lip} with the right-hand side scaled by $\lambda$.

Recalling the notion of $\kappa$-directedness from \Cref{def:gen}, we remind the reader of \cite[Definition 1.17]{ar}:

\begin{definition}
  Let $\kappa$ be a regular cardinal and $\cC$ a category.
  \begin{itemize}
  \item An object $x\in \cC$ is {\it $\kappa$-presentable} if $\mathrm{hom}_{\cC}(x,-)$ preserves $\kappa$-directed colimits.
  \item $\cC$ is {\it locally $\kappa$-presentable} if it is cocomplete and every object is a $\kappa$-directed colimit of $\kappa$-presentable objects.
  \item Finally, $\cC$ is {\it locally presentable} if it is locally $\kappa$-presentable for some regular cardinal $\kappa$.
  \end{itemize}
\end{definition}

As observed in \cite[Examples 2.3 (1) and (2)]{ar-ap}, \cat{Met} and \cat{CMet} are both locally $\aleph_1$-presentable.

We follow
\begin{itemize}
\item \cite[\S 6]{dr} in writing ${\bf 2}_{\delta}$ for the two-point space $\{x,x'\}$ with $d(x,x')=\delta\in \bR_{\ge 0}\cup\{\infty\}$;
\item and \cite[Example 2.3 (1)]{ar-ap} in referring to metric spaces all of whose pairwise distances are infinite as {\it discrete}.
\end{itemize}

\section{Convex metric spaces}\label{se:cvx}

In order to avoid some slightly bothersome corner cases (e.g. the issue of whether or not the two-point space ${\bf 2}_{\infty}$ is convex) we depart from \cite[\S 6]{dr} slightly in what is meant by `convex':

\begin{definition}\label{def:cvx}
  A metric space $(X,d)\in \cat{CMet}$ is {\it convex} if for every $x\ne y\in X$ with $d(x,y)<\infty$ there is some $z\in X$ distinct from both $x$ and $y$ such that
  \begin{equation*}
     d(x,y) = d(x,z) + d(z,y). 
  \end{equation*}
\end{definition}

In other words, we only require such ``intermediate'' points $z$ for $x,y\in X$ a {\it finite} distance apart. This also conflicts slightly with the notion introduced in \cite[Definition 3.6.5]{bbi}, where convexity automatically entails (by definition) the finiteness of the metric.

Per the discussion in \cite[Remark 6.9]{dr}, \cat{CCMet} is symmetric monoidal with the tensor product $(X,d_X)\otimes (Y,d_Y)$ given by the Cartesian product $X\times Y$ as a set, together with the $\ell^1$ metric:
\begin{equation*}
  d_{X\otimes Y}((x,y),(x',y')) := d_X(x,x') + d_Y(y,y'),\ \forall x,x'\in X,\ \forall y,y'\in Y.
\end{equation*}

It is a natural question (asked in passing in loc.cit.) whether this monoidal structure is closed. The existence of an internal hom object
\begin{equation*}
  [X,Y]\in \cat{CCMet}
\end{equation*}
makes sense for each pair of objects $(X,d_X)$ and $(Y,d_Y)$ in \cat{CCMet}: by definition, it would be the object defining the contravariant functor
\begin{equation*}
  \cat{CCMet}(-\otimes X,Y):\cat{CCMet}^{op}\to \cat{Set}.
\end{equation*}
Naturality in $X$ or $Y$, when these objects exist, follows from this characterization and Yoneda (e.g. \cite[Corollary 6.19]{ahs}).

\cite[\S 6]{dr} also considers the category $\cat{CMet}$ of complete metric spaces (i.e. \cat{CCMet} sans convexity). It is monoidal closed, with
\begin{equation}\label{eq:v0int}
  [X,Y]_{\cV_0} \cong (\cV_0(X,Y),d_{\sup}):
\end{equation}
see \cite[Remark 2.2 and Example 2.3 (2)]{ar-ap}. Since $\cat{CCMet}$ is monoidal and full in $\cat{CMet}$, the following simple remark tells us how the respective internal homs would relate to one another.

\begin{lemma}\label{le:corefl}
  Let $\cV\subseteq \cV_0$ be a full monoidal category of a monoidal closed category. For objects $X,Y\in \cV$, the internal hom $[X,Y]_{\cV}$ exists if and only if the object $[X,Y]_{\cV_0}\in \cV_0$ has a coreflection in $\cV$, and in that case $[X,Y]_{\cV}$ is that coreflection.
\end{lemma}
\begin{proof}
  Indeed, $[X,Y]_{\cV}$ would have precisely the same universal property as the $\cV$-coreflection of $[X,Y]_{\cV_0}$: representing the contravariant $\cat{Set}$-valued functor
  \begin{equation*}
    \cV(-\otimes X, Y)\cong \cV_0(-\otimes X, Y)\cong \cV_0(-,[X,Y]_{\cV_0})
  \end{equation*}
  on $\cV$.
\end{proof}
As in the above equation, we occasionally decorate the internal hom by the category where it is intended to live: $[X,Y]$ is also $[X,Y]_{\cV}$.

We will need some more metric-geometry vocabulary, for which we refer to \cite[Chapter 1]{grm} and \cite[Chapter 2]{bbi}. A small amount of care is needed in adapting some statements from the former source, where `metric space' has the more conventional meaning allowing only for finite metrics \cite[Introduction]{grm}. 

First, as we recall shortly, the distance of a complete convex metric space can, in a sense, be recovered from non-expansive paths in the space. The relevant notions follow (\cite[Definitions 1.2 and 1.7]{grm} or \cite[Definitions 2.1.6, 2.1.10 and 2.3.1]{bbi}).

\begin{definition}\label{def:lspace}
  Let $(X,d)$ be a metric space.
  \begin{itemize}
  \item The {\it length} $\ell(f)$ of a continuous curve $f:[a,b]\to X$ is
    \begin{equation*}
      \ell(f):=\sup \sum_{i=0}^n d(f(t_i),f(t_{i+1})),
    \end{equation*}
    where the supremum is taken over all selections of intermediate points
    \begin{equation*}
      a=t_0\le t_1\le \cdots\le t_n = t_{n+1}=b.
    \end{equation*}
  \item A curve $f:[a,b]\to X$ is {\it rectifiable} if $\ell(f)<\infty$.
  \item The {\it path} or {\it length} or {\it intrinsic metric} $d_{\ell}$ attached to $d$ is
    \begin{equation*}
      d_{\ell}(x,y):=\inf\ell(f),\ f:[a,b]\to X,\ f(a)=x\text{ and }f(b)=y.
    \end{equation*}
  \item $(X,d)$ is a {\it path} (or {\it length}, or {\it intrinsic}) {\it metric space} if $d=d_{\ell}$.
  \item A length metric space $(X,d)$ is {\it strict} (and its metric is {\it strictly intrinsic}) if any two points $x,x'$ with $d(x,x')<\infty$ can be connected by a path of length $d(x,x')$.
  \end{itemize}
\end{definition}

\begin{remarks}\label{res:dl}
  \begin{enumerate}[(1)]
  \item\label{item:1} It is immediate from the definition of $d_{\ell}$ that $d\le d_{\ell}$, but in general the inequality is strict. Indeed, as observed in \cite[Example 1.4 (a)]{grm}, even the topologies induced by the two metrics are generally distinct: path components in the $d$-topology are {\it clopen} (both closed and open) in the $d_{\ell}$-topology.

    This same class of examples also shows that even when $d$ takes only finite values, $d_{\ell}$ might not: $d_{\ell}(x,y)=\infty$ whenever $x$ and $y$ are in different path components.
  \item\label{item:2} \cite[Remark following Proposition 1.6]{grm} notes that for any $(X,d)$, the resulting metric space $(X,d_{\ell})$ is in fact a path metric space because the construction $d\mapsto d_{\ell}$ is idempotent:
    \begin{equation*}
      (d_{\ell})_{\ell}=d_{\ell}
    \end{equation*}
    We take this for granted implicitly below.
  \item\label{item:3} It is not difficult to see that if $(X,d)$ is complete then so is $(X,d_{\ell})$: see the proof of \Cref{pr:whencorefl}.
  \item\label{item:10} Our notion of `strictly intrinsic' is somewhat weaker than that of \cite[Definition 2.1.10]{bbi}; the latter automatically implies that all distances are finite.

    Indeed, loc.cit. asks that {\it any} two points $x$ and $x'$ be the endpoints of a continuous map from an interval (of length precisely $d(x,x')$, but this is beside the point here). If $d(x,x')=\infty$ then $x$ and $x'$ lie in distinct clopen components in the topology induced by $d$, so this cannot happen.

    Some of the results in \cite{bbi} seem to ignore the issue of infinite distances, so that some care is required in applying them to generalized metric spaces: for \cite[Theorem 2.4.16, part 1.]{bbi} to hold, for instance,
    \begin{itemize}
    \item one must assume the metric is finite;
    \item or extend the notion of `strictly intrinsic' to possibly $\infty$-valued metrics, as in the present definition;
    \item in which case, for \cite[Lemma 2.4.8]{bbi} to hold, one would have to also modify \cite[Definition 2.4.7]{bbi} of {\it midpoints} by requiring the defining constraint only for finite-distance pairs (as in \Cref{re:cvxchar}).
    \end{itemize}
  \end{enumerate}  
\end{remarks}

As we are working with categories of metric spaces where morphisms are {\it contractive}, it is perhaps worth noting the following alternative description of the $d\mapsto d_{\ell}$ construction.

\begin{lemma}\label{le:altdl}
  For a metric space $(X,d_X)$ and points $x,x'\in X$ the intrinsic metric $d_{X,\ell}$ can be recovered as
  \begin{equation}\label{eq:1lipl}
    d_{X,\ell}(x,x')
    =
    \inf\{\ell\in\bR_{\ge 0}\ |\ \exists \text{ 1-Lipschitz }\varphi:[0,\ell]\to (X,d_X),\ \varphi(0)=x,\ \varphi(\ell)=x'\}.
  \end{equation}
\end{lemma}
\begin{proof}
  Consider a rectifiable curve
  \begin{equation*}
    f:[a,b]\to X,\ a\mapsto x,\ b\mapsto x'.
  \end{equation*}
  By \cite[Proposition 2.5.9]{bbi} it decomposes as $f=\varphi\circ \alpha$ for non-decreasing $\alpha:[a,b]\to [0,\ell(f)]$ and an {\it arc-length-parametrized} (\cite[Definition 2.5.7 and discussion following Remark 2.5.8]{bbi})
  \begin{equation*}
    \varphi:[0,\ell(f)]\to X,\ 0\mapsto x,\ \ell(f)\mapsto x'.
  \end{equation*}
  We now have $\ell(f)=\ell(\varphi)$ (i.e. composition with a non-decreasing map makes no difference to the length), and $\varphi$ is 1-Lipschitz.
\end{proof}

And an immediate consequence that we will take for granted repeatedly in the sequel:

\begin{corollary}\label{cor:strictiff}
  A length metric space is strict in the sense of \Cref{def:lspace} if and only if all finite infima \Cref{eq:1lipl} are achieved (i.e. are actual minima).  \qedhere
\end{corollary}

It is a classical result of Menger's \cite{menger} that complete convex metric spaces are length metric spaces. Much more is true though; before stating the full result, recall (\cite[\S 2.5]{kk}, \cite[Definition 14.2]{blm-met} and \cite[Definition 1.9]{grm}):

\begin{definition}\label{def:msegm}
  Let $(X,d)$ be a metric space.

  A {\it metric segment} or {\it minimizing geodesic} in $X$ is an isometry $f:[a,b]\to X$ from a finite interval (with its usual distance function) to $X$.

  We refer to $f(a)$ and $f(b)$ as the {\it endpoints} of the segment or say that the segment (geodesic) {\it connects} them.
\end{definition}

Menger's theorem, referred to above, says that not only are complete convex metric spaces path metric spaces, but in fact, for any two points $x,y$ (with $d(x,y)<\infty$ in our present context of {\it generalized} metric spaces), there is a minimizing geodesic connecting $x$ and $y$; see for instance \cite[unnumbered Theorem preceding Lemma 2.1]{gk-met} or \cite[Theorem 14.1]{blm-met} for proofs (the result also appears as \cite[Theorem 2.16]{kk}).

\begin{remark}\label{re:cvxchar}
  We have now come full-circle back to intrinsic metrics: for a complete metric space $(X,d)$, the following are equivalent:
  \begin{enumerate}[(a)]
  \item\label{item:11} convexity;
  \item\label{item:12} $(X,d)$ is strictly intrinsic in the sense of \Cref{def:lspace};
  \item\label{item:13} any two $x,x'\in X$ with $d(x,x')<\infty$ have a {\it midpoint}: a point $y$ with
    \begin{equation*}
      d(x,y) = d(x',y) = \frac{d(x,x')}2.
    \end{equation*}
  \end{enumerate}
  Indeed, \Cref{item:11} implies \Cref{item:12} by Menger, while \Cref{item:12} $\Rightarrow$ \Cref{item:13} and \Cref{item:13} $\Rightarrow$ \Cref{item:11} are clear.
\end{remark}

\Cref{le:corefl} suggests that we should study coreflections of $\cat{CMet}$ in $\cat{CCMet}$. The following result describes the circumstances when these exist.

\begin{proposition}\label{pr:whencorefl}
  Consider an object $(X,d_X)\in \cV_0:=\cat{CMet}$.
  \begin{enumerate}[(1)]
  \item\label{item:6} The intrinsic metric $d:=d_{X,\ell}$ of \Cref{def:lspace} is a complete generalized metric on $X$.
  \item\label{item:7} $X$ has a coreflection in $\cV:=\cat{CCMet}$ precisely when $(X,d)$ is strict in the sense of \Cref{def:lspace}, in which case $(X,d)\in \cV$ is the coreflection.
  \end{enumerate}
\end{proposition}
\begin{proof}
  We tackle the claims in turn.

  {\bf \Cref{item:6}} The triangle inequality follows from the fact that 1-Lipschitz maps
  \begin{equation*}
    [0,\ell]\to X,\quad [0,\ell']\to X
  \end{equation*}
  ending and respectively starting at the same point splice together to a 1-Lipschitz map defined on $[0,\ell+\ell']$. The non-degeneracy condition
  \begin{equation*}
    d(x,y)=0\Rightarrow x=y
  \end{equation*}
  being obvious (for instance because $d$ dominates $d_X$), we do indeed have a generalized metric. As for completeness: note first that a $d$-Cauchy sequence $(x_n)_n$ is certainly $d_X$-Cauchy, because $d\ge d_X$. Such a sequence will thus converge to some $x\in X$ in the original $d_X$ metric. We can now find positive integers
  \begin{equation*}
    n_0<n_1<\cdots
  \end{equation*}
  such that $d(x_{n_{k-1}},x_{n_{k}})<\frac 1{4^k}$ for $k\ge 1$. We thus have 1-Lipschitz curves
  \begin{equation*}
    \left[\frac 14+\cdots+\frac 1{4^{k-1}},\quad \frac 14+\cdots+\frac 1{4^k}\right]\to X
  \end{equation*}
  connecting $x_{n_{k-1}}$ and $x_{n_{k}}$ respectively. For fixed $k$ those with indices $k$ and higher splice together to a 1-Lipschitz curve
  \begin{equation*}
    \left[\frac 14+\cdots+\frac 1{4^{k-1}},\quad \frac 13\right]\to X
  \end{equation*}
  connecting $x_{n_{k-1}}$ and $x$, whence the conclusion that $x_{n_k}\to x$ in the $d$-topology.
  
  {\bf \Cref{item:7}} If ${\bf 1}\in \cV$ is the one-point space (and hence the monoidal unit of both $\cV_0$ and $\cV$) then the functors $\cV({\bf 1},-)$ and $\cV_0({\bf 1},-)$ are both forgetful to $\cat{Set}$. It follows from this that a coreflection of $(X,d_X)\in \cV_0$ in $\cV$ must be of the form
  \begin{equation}\label{eq:crflid}
    \id:(X,d')\to (X,d_X)
  \end{equation}
  for some alternative distance $d'\ge d_X$, to be determined (when it exists). On to the two implications that constitute claim \Cref{item:7}.

  {\bf ($\Leftarrow$)} We already know from part \Cref{item:6} that \Cref{eq:1lipl} is a complete generalized metric. Note furthermore that any contraction $f:(Y,d_Y)\to (X,d_X)$ with $Y$ convex factors through a contraction to $(X,d)$ with $d$ as in \Cref{eq:1lipl}: any two points $y_i\in Y$, $i=0,1$ are (by \cite[Theorem 2.16]{kk}) the endpoints of a metric segment
  \begin{equation*}
    \gamma:[0,d_Y(y_0,y_1)]\to Y,
  \end{equation*}
  so we have a 1-Lipschitz curve
  \begin{equation*}
    \varphi:=f\circ\gamma:[0,d_Y(y_0,y_1)]\to X
  \end{equation*}
  with $\varphi(0)=x_0:=f(y_0)$ and $\varphi(d_Y(y_0,y_1))=x_1:=f(y_1)$. It follows, then, that
  \begin{equation}\label{eq:dyged}
    d_Y(y_0,y_1)\ge d(x_0,x_1)
  \end{equation}
  for the distance $d$ of \Cref{eq:1lipl}. This shows that $\id:(X,d)\to (X,d_X)$ will indeed be a coreflection, provided $(X,d)$ is convex. 

  If $\ell:=d(x,x')<\infty$ then a 1-Lipschitz map
  \begin{equation*}
    [0,\ell]\to X,\ 0\mapsto x,\ \ell\mapsto x'
  \end{equation*}
  (assumed to exist by the infimum-realization hypothesis) must in fact be a metric segment, hence convexity.

  {\bf ($\Rightarrow$)} If a coreflection exists, we have already noted it must be of the form \Cref{eq:crflid} for some metric $d'$. It remains to argue that the infima \Cref{eq:1lipl} are achieved when finite and that \Cref{eq:1lipl} is the distance function on the coreflection.

  To that end, let $x,x'\in X$ with $d(x,x')<\infty$ meaning simply that there are 1-Lipschitz curves connecting $x$ and $y$. Any such curve with domain $[0,\ell]$ will factor through \Cref{eq:crflid} and hence $\ell\ge d'(x,x')$. But then the infimum \Cref{eq:1lipl} also dominates $d'(x,x')$; the opposite inequality was noted above, in the proof of ($\Leftarrow$) (see \Cref{eq:dyged}), so that $d'=d$.

  Finally, the fact that the infimum is in fact achieved then follows from Menger's \cite[Theorem 2.16]{kk} again: every $x,x'\in X$ with $d(x,x')<\infty$ are the endpoints of a metric segment of length $d(x,x')$.
\end{proof}

We now have the following description of (potential) internal homs in $\cat{CCMet}$.

\begin{corollary}\label{cor:ihomchar}
  Let $(X,d_X)$ and $(Y,d_Y)$ be two objects in $\cV:=\cat{CCMet}$.
  \begin{enumerate}[(1)]
  \item\label{item:4} If $[X,Y]\in\cV$ exists, then it must be $\cV(X,Y)$ equipped with the following metric:
    \begin{equation}\label{eq:minach}
      d(f,g):=\min\{\ell\ |\ \exists \text{ 1-Lipschitz }\varphi:[0,\ell]\to (\cV(X,Y),d_{\sup}),\ \varphi(0)=f,\ \varphi(\ell)=g\},
    \end{equation}
    where
    \begin{equation*}
      d_{\sup}(f,g):=\sup_{x\in X}d_Y(f(x),g(x)). 
    \end{equation*}
    In particular, the existence of the internal hom requires that the minimum be achieved whenever the infimum is finite.
  \item\label{item:5} Conversely, if the minima \Cref{eq:minach} are achieved for arbitrary $f,g\in \cV(X,Y)$ for which the respective infimum is finite, then \Cref{eq:minach} defines a generalized metric on $\cV(X,Y)$ making it into the internal hom.
  \end{enumerate}
\end{corollary}
\begin{proof}
  This is an immediate application of \Cref{le:corefl}, \Cref{pr:whencorefl} and the description \Cref{eq:v0int} of internal homs in $\cat{CMet}$.
\end{proof}

Some preparatory remarks follow, aimed at giving sufficient conditions for the existence of internal homs in \cat{CCMet}.

\begin{proposition}\label{pr:dl1cpct}
  Let $(X,d_X)\in \cV_0:=\cat{CMet}$ be a complete generalized metric space. If the finite-radius closed balls of $X$ are compact, then 
  \begin{enumerate}[(a)]
  \item\label{item:8} the same holds for the internal hom $[Y,X]_{\cV_0}$ for any $(Y,d_Y)\in \cat{CMet}$;  
  \item\label{item:9} and $(X,d_X)$ satisfies the condition in \Cref{pr:whencorefl} \Cref{item:7}, and thus it has a coreflection in $\cV:=\cat{CCMet}$.
  \end{enumerate}
\end{proposition}
\begin{proof}
  The arguments are very similar, and both rely on {\it Ascoli's theorem} (\cite[Theorem 47.1]{mnk}) characterizing relatively compact spaces of maps in the {\it compact-open topology} \cite[Definition preceding Theorem 46.8]{mnk}.
  
  {\bf \Cref{item:8}} As recalled in \Cref{eq:v0int}, the internal hom is simply the space of 1-Lipschitz maps with the supremum norm. For such a 1-Lipschitz map $f:Y\to X$ and $r\in \bR_{>0}$, the radius-$r$ ball $B\subset [Y,X]_{\cV_0}$ around $f$ is an {\it equicontinuous family} \cite[Definition preceding Lemma 45.2]{mnk} and for each $y\in Y$ the set
  \begin{equation*}
    \{f'(y)\ |\ f'\in B\}
  \end{equation*}
  is contained in a finite-radius ball of $X$ and is thus relatively compact by assumption. The relative compactness of $B$ in the compact-open topology on
  \begin{equation*}
    \cat{cont}(Y\to X)
  \end{equation*}
  now follows from Ascoli's theorem. Clearly, though, closed balls in $[Y,X]_{\cV}$ are also closed in the compact-open topology (indeed, even in the {\it point}-open topology of \cite[Definition preceding Theorem 46.1]{mnk}, which is weaker).
   
  {\bf \Cref{item:9}} Writing $d:=d_{X,\ell}$ for brevity, we have to argue that 
  \begin{equation*}
    d(x,x') :=
    \inf\{\ell\ |\ \exists \text{ 1-Lipschitz }\varphi:[0,\ell]\to (X,d_X),\ \varphi(0)=x,\ \varphi(\ell)=x'\}
  \end{equation*}
  is achieved as an actual minimum whenever it is finite.

  Suppose $d(x,x')=r\in \bR_{>0}$, and hence we have 1-Lipschitz curves
  \begin{equation*}
    [0,\lambda r]\to X,\ 0\mapsto x,\ \lambda r\mapsto x'
  \end{equation*}
  for $\lambda>1$ arbitrarily close to $1$. Rescaling, this means $\lambda$-Lipschitz curves
  \begin{equation*}
    \gamma_{\lambda}:[0,r]\to X,\ \gamma_{\lambda}(0)=x,\ \gamma_{\lambda}(r)=x'.
  \end{equation*}
  The family $\{\gamma_{\lambda}\}_{\lambda}$ is equicontinuous by the $\lambda$-Lipschitz condition, and each set
  \begin{equation*}
    \{\gamma_{\lambda}(t)\ |\ \lambda\}\quad\text{for}\quad t\in [0,r]
  \end{equation*}
  is contained in a (compact, by assumption) finite-radius ball in $X$. It follows that $\{\gamma_{\lambda}\}_{\lambda}$ is relatively compact, and as $\lambda\searrow 1$ some subnet will converge to a 1-Lipschitz map
  \begin{equation*}
    [0,r]\to X,\ 0\mapsto x,\ r\mapsto x'.
  \end{equation*}
  This finishes the proof, the coreflection claim being a consequence of \Cref{pr:whencorefl} \Cref{item:7}.
\end{proof}

\begin{remark}
  The intrinsic space $(X,d_{X,\ell})$ of \Cref{def:lspace} will not, in general, have compact finite-radius closed balls, even if $X$ does.

  \cite[Example 1.4 (b$_+$)]{grm} illustrates this phenomenon. One first recovers the standard topology on $X:=\bR^n$ from the metric
  \begin{equation*}
    d_X(x_1,x_2):=|r_1-r_2| + \min(r_i)\|s_1-s_2\|^{\frac 12},
  \end{equation*}
  where
  \begin{equation*}
    x_i=t_i s_i,\ t_i\in \bR_{\ge 0},\ s_i\in \bS^{n-1}
  \end{equation*}
  are the respective polar-coordinate descriptions of $x_i$. It is not difficult to see then that $d_{X,\ell}$ coincides with the path metric $d_{X,\ell}$. As \cite[Example 1.4 (b$_+$)]{grm} observes though, the latter induces on $\bR^n$ the strongest topology for which the ray embeddings
  \begin{equation*}
    \iota_s:\bR_{\ge 0}\ni t\mapsto ts\in \bR^n,\quad s\in \bS^{n-1}
  \end{equation*}
  are continuous.

  In $(X,d_{X,\ell})$ the unit sphere $\bS^{n-1}$ is discrete, infinite, and contained in the closed unit ball around the origin.
\end{remark}

We can now state the aforementioned sufficiency result for internal-hom existence.

\begin{theorem}\label{th:lcexists}
  Let $(X,d_X)$ and $(Y,d_Y)$ be two objects in \cat{CCMet}. If either of the following equivalent conditions holds then the internal hom $[Y,X]\in\cat{CCMet}$ exists:
  \begin{itemize}
  \item $X$ is locally compact;
  \item closed balls are compact in $X$.
  \end{itemize}
\end{theorem}
\begin{proof}
  $X$ is a path metric space in the sense of \Cref{def:lspace} by \cite[Theorem 2.16]{kk}, so indeed local compactness is equivalent to closed balls being compact, by the {\it Hopf-Rinow theorem} (\cite[following Definition 1.9]{grm} or \cite[Theorem 2.5.28]{bbi}).
  
  Since $X$ is convex it is not a two-point space, so neither is $\cV(Y,X)$. The two points of \Cref{pr:dl1cpct} show that $[Y,X]_{\cV_0}$ has a coreflection in $\cV$, and \Cref{le:corefl} finishes the proof. 
\end{proof}

\subsection{On and around gluing}\label{subse:glue}

In order to argue that (\Cref{th:lcexists} notwithstanding) the category \cat{CCMet} is {\it not} closed, we construct examples via the {\it gluing} procedure outlined in \cite[\S 3.1.2]{bbi}.

\begin{definition}\label{def:glue}
  \begin{enumerate}[(a)]
  \item\label{item:14} Let $(X,d)$ be a metric space and `$\sim$' an equivalence relation on $X$. Define the {\it quotient semi-metric} (\cite[Definition 3.1.12]{bbi}) $d_{\sim}$ on $X/\sim$ by
    \begin{equation*}
      d_{\sim}(x,x')
      :=
      \inf \sum_{s=0}^n d(p_s,q_s),
    \end{equation*}
    with the infimum taken over all tuples with
    \begin{equation*}
      p_0 = x,\quad q_n = x'\quad\text{and}\quad q_s\sim p_{s+1}\quad \forall\ 0\le s\le n-1. 
    \end{equation*}    
    The {\it quotient metric space} $(X/d_{\sim},d_{\sim})$ (or the metric space obtained by {\it gluing $(X,d)$ along `$\sim$'}) is constructed from the semi-metric $d_{\sim}$, by identifying pairs of points with zero distance.
  \item\label{item:15} Consider metric spaces $(X_i,d_i)$, $i\in I$ and $(X,d)$ equipped with contractions
    \begin{equation*}
      \iota_i:(X,d)\to (X_i,d_i).
    \end{equation*}
    The metric space
    \begin{equation*}
      \coprod_{X}X_i\quad\text{or}\quad \coprod_{\iota_i}X_i
    \end{equation*}
    obtained by {\it gluing $X_i$ along $X$} is constructed by
    \begin{itemize}
    \item first forming the disjoint union $\coprod_i (X_i,d_i)$ equipped with the original distances $d_i$ on the individual $X_i$ and infinite distances across distinct $X_i$ (as in \cite[Definition 3.1.15]{bbi});
    \item and then gluing that disjoint union as in item \Cref{item:14}, along the relation with equivalence classes
      \begin{equation*}
        \{\iota_i(x)\ |\ i\in I\}\ \text{for}\ x\in X.
      \end{equation*}
    \end{itemize}
  \end{enumerate}
\end{definition}

We use points $p_s$ and $q_s$ as in \Cref{def:glue} frequently, so it will be handy to have a term for the notion.

\begin{definition}\label{def:simchn}
  For an equivalence relation `$\sim$' on a space $X$ and points $x,x'\in X$ a {\it $\sim$-chain connecting $x$ and $x'$} (or just `chain' when the relation is understood) is a sequence of pairs $p_s,q_s\in X$, $0\le s\le n$ with
  \begin{equation*}
    p_0=x,\ q_n=x',\ q_s\sim p_{s+1},\quad \forall\ 1\le i\le n-1. 
  \end{equation*}
\end{definition}

We will have to construct certain pathological convex metric spaces by gluing, which requires that said glued spaces be strictly intrinsic (\Cref{re:cvxchar}). While \cite[paragraph following Exercise 3.1.13]{bbi} notes that the property of being intrinsic survives gluing, examples are easily produced of {\it strictly} intrinsic spaces which glue to non-strictly intrinsic quotients:

\begin{example}\label{ex:nstr}
  Consider the family
  \begin{equation*}
    (X_\varepsilon,d_{\varepsilon}) := [0,1+\varepsilon],\ \varepsilon>0
  \end{equation*}
  with their usual interval distances, glued along the embeddings
  \begin{equation*}
    \iota_{\varepsilon}:{\bf 2}_{\infty}\to X_{\varepsilon}
  \end{equation*}
  identifying the two-point space with the endpoints of said intervals. The result
  \begin{equation*}
    (X,d) := \coprod_{\iota_{\varepsilon}}(X_{\varepsilon},d_{\varepsilon})
  \end{equation*}
  of the gluing procedure consists of length-$(1+\varepsilon)$ intervals with a common pair of endpoints a distance of $1$ apart. Said points are not connected by any curves of length precisely $1$ though, by construction: the curves $[0,1+\varepsilon]\to X$ have respective lengths $1+\varepsilon$.
\end{example}

In the category \cat{Met} of (possibly non-complete) metric spaces with contractions the pushout of a pair
\begin{equation*}
  j_i:Y\to X_i,\ i=1,2
\end{equation*}
is nothing but the disjoint union $X_1\coprod X_2$ glued along the relation identifying
\begin{equation*}
  j_1(y)\sim j_2(y),\ y\in Y.
\end{equation*}
By contrast, a pushout
\begin{equation}\label{eq:pushcmet}
  \begin{tikzpicture}[auto,baseline=(current  bounding  box.center)]
    \path[anchor=base] 
    (0,0) node (l) {$Y$}
    +(2,.5) node (u) {$X_1$}
    +(2,-.5) node (d) {$X_2$}
    +(4,0) node (r) {$X$}
    ;
    \draw[->] (l) to[bend left=6] node[pos=.5,auto] {$\scriptstyle j_1$} (u);
    \draw[->] (u) to[bend left=6] node[pos=.5,auto] {$\scriptstyle \iota_1$} (r);
    \draw[->] (l) to[bend right=6] node[pos=.5,auto,swap] {$\scriptstyle j_2$} (d);
    \draw[->] (d) to[bend right=6] node[pos=.5,auto,swap] {$\scriptstyle \iota_2$} (r);
  \end{tikzpicture}
\end{equation}
in \cat{CMet} is the {\it completion} of that glued space. The qualification is crucial, as the gluing alone need not produce a complete space:

\begin{example}\label{ex:gluenotcmpl}
  Let
  \begin{itemize}
  \item $X_1$ be the disjoint union of intervals $[\ell_{2n},r_{2n}]$ of respective lengths $\frac 1{2^{2n}}$ for $n\in \bZ_{\ge 0}$ (with infinite distances between points on distinct intervals);
  \item $X_2$ be the disjoint union of intervals $[\ell_{2n+1},r_{2n+1}]$ of respective lengths $\frac 1{2^{2n+1}}$ for $n\in \bZ_{\ge 0}$;
  \item $Y$ a countable discrete metric space;
  \item and $j_i$, $i=1,2$, respectively, the identifications of $Y$ with
    \begin{itemize}
    \item the endpoints of the ``even intervals'', minus the leftmost:
      \begin{equation*}
        r_{0},\ \ell_2,\ r_2,\ \ell_4,\ r_4,\ \cdots
      \end{equation*}
    \item the endpoints of the ``odd intervals'':
      \begin{equation*}
        \ell_1,\ r_1,\ \ell_3,\ r_3,\ \cdots
      \end{equation*}
    \end{itemize}
  \end{itemize}
  The glued space $X_1\coprod_Y X_2$ is the splicing together of all of the intervals, consecutively, alternating between even and odd. It is, in short, a (non-complete) half-open interval of length 2.
\end{example}

It will be convenient to glue only under circumstances that avoid the issues exhibited by \Cref{ex:gluenotcmpl}, in that completeness is automatic. This entails imposing some constraints on the maps $j_i$ in a pushout diagram \Cref{eq:pushcmet}. 

\begin{definition}\label{def:cexp}
  For a constant $C>0$, a map $f:(X,d_X)\to (Y,d_Y)$ between two metric spaces is {\it $C$-expansive} if
  \begin{equation*}
    d_Y(fx,fx')\ge C d_X(x,x'),\ \forall x,x'\in X.
  \end{equation*}  
\end{definition}

\begin{remark}
  The Lipschitz maps (as most of ours are) that are also $C$-expansive are precisely the {\it bi-Lipschitz} maps of \cite[Definition 1.4.6]{bbi}.
\end{remark}

We refer to colimits of diagrams
\begin{equation}\label{eq:jis}
  j_i:Y\to X_i,\ i\in I
\end{equation}
consisting of common-source arrows as `pushouts', even when the family consists of more than two arrows. In the language of \cite[Exercise 11L]{ahs}, say, these would be {\it multiple} pushouts, but context should serve as sufficient guard against ambiguity.

It will be convenient to set up some language and conventions for handling such pushouts and their colimits. Denote by
\begin{equation*}
  X:=\coprod_{Y,i} X_i
\end{equation*}
the coproduct in \cat{Met}, i.e. the gluing of the disjoint union $\coprod_i X_i$ along the relation identifying, for each $y\in Y$, all $j_i(y)\in X_i$. The maps \Cref{eq:jis} will typically be one-to-one for us, and hence so will the canonical contractions $\iota_i:X_i\to X$. For that reason, we occasionally identify $X_i$ with its image $\iota_i(X_i)\subseteq X$.

By definition, for points $x$ and $x'$ in $X$ the distance $d_X(x,x')$ is the infimum of the sums
\begin{equation}\label{eq:psqs}
  \sum_{s=0}^n d(p_s,q_s)
\end{equation}
for $\sim$-chains $(p_s,q_s)_s$ as in \Cref{def:simchn}, where
\begin{equation*}
  j_i(y)\sim j_{i'}(y),\ \forall y\in Y,\ \forall i,i'\in I. 
\end{equation*}
It is harmless to make a number of simplifying assumptions on the $\sim$-chains in question.

\begin{definition}\label{def:strmchn}
  Let \Cref{eq:jis} be morphisms in \cat{Met} with the induced relation `$\sim$' on $\coprod_i X_i$. A $\sim$-chain $(p_s,q_s)_{s=0}^n$ connecting $p_0=x$ and $q_n=x'$ is {\it streamlined} if
  \begin{itemize}
  \item for each $s$, the points $p_s$ and $q_s$ are a finite distance apart in $\coprod X_i$ (and in particular belong to the same $X_i$). For chains {\it not} satisfying this condition the sum \Cref{eq:psqs} is infinite, so they contribute nothing to the infimum. 
  \item no $q_s$ is {\it equal} to the point $p_{s+1}$ (to which it is supposed to be {\it equivalent}). Indeed, otherwise we can always collapse the chain to a shorter one without increasing \Cref{eq:psqs}: if $q_s=p_{s+1}$ then
    \begin{equation*}
      d(p_s,q_s) + d(p_{s+1},q_{s+1})
      =
      d(p_s,q_s) + d(q_s,q_{s+1})
      \le
      d(p_s,q_{s+1}). 
    \end{equation*}
  \end{itemize}
  In particular, this second condition ensures that the ``intermediate'' points
  \begin{equation*}
    q_0,\ p_1,\ q_1,\ \cdots,\ p_n
  \end{equation*}
  all belong to images $j_i(Y)\subseteq X_i$: they must be equivalent to points they are not equal to.
\end{definition}

The following result will be helpful in constructing various examples with desired properties by gluing, while making sure that completeness and the intrinsic property are preserved.

\begin{theorem}\label{th:largepush}
  Let $1\ge C>0$ and consider a family
  \begin{equation*}
    j_i:Y\to X_i,\ i\in I
  \end{equation*}
  of $C$-expansive morphisms in \cat{CMet}.
  \begin{enumerate}[(a)]
  \item\label{item:22} The canonical maps
    \begin{equation*}
      \iota_i:X_i\to X:=\coprod_{Y,i}X_i
    \end{equation*}
    into the pushout in \cat{Met} are $C$-expansive.
  \item\label{item:24} The distance $d_X$ between points in the image of $Y$ through the canonical map
    \begin{equation*}
      \iota:=\iota_i\circ j_i:Y\to X\quad\text{(arbitrary $i$)}
    \end{equation*}
    can be computed as
    \begin{equation}\label{eq:precdx}
      d_X(\iota y,\iota y') = \inf_{i\in I}d_{X_i}(j_i y,\ j_i y'),\ y,y'\in Y. 
    \end{equation}
    In particular, $\iota:Y\to X$ is also $C$-expansive.
  \item\label{item:21} More generally, for $x\in X_i$ and $x'\in X_{i'}$ with $i\ne i'$ we have
    \begin{equation}\label{eq:dii}
      d_X(\iota_i x,\ \iota_{i'} x') = \inf_{y,y'\in Y} \left(d_{X_i}(x,j_i y) + d_X(\iota y,\iota y') + d_{X_{i'}}(j_{i'}y',x')\right). 
    \end{equation}
  \item\label{item:27} For $x,x'\in X_i$ the distance $d_X(\iota_i x,\iota_i x')$ is the smallest of \Cref{eq:dii} (with $i'=i$) and $d_{X_i}(x,x')$.
  \item\label{item:23} The space $X$ is automatically complete, and hence also the \cat{CMet}-pushout of the $j_i$.
  \end{enumerate}
\end{theorem}
\begin{proof}
  We prove the statements in the order in which they were made.
  \begin{enumerate}[(a)]
  \item\label{item:25} Consider two points $x,x'$ in some $X_{i_0}$, for $i_0\in I$ arbitrary but fixed throughout the proof, and a streamlined chain (\Cref{def:strmchn}) $(p_s,q_s)_{s=0}^n$ connecting $\iota_{i_0} x$ and $\iota_{i_0} x'$.
    
    The intermediate pair $(p_s,q_s)$, $0<s<n$ is contained in, say, $X_i$, with
    \begin{equation*}
      p_s = j_i(y_s),\ q_s=j_i(y_{s+1}),\ y_{\bullet}\in Y
    \end{equation*}
    (that there are such $y_{\bullet}$ follows from the assumption that the chain is streamlined). We then have
    \begin{align*}
      d(p_s,q_s) &= d_{X_i}(p_s,q_{s})\\
                 &= d_{X_i}(j_i(y_s), j_i(y_{s+1}))\\
                 &\ge C d_Y(y_s,y_{s+1})
                   \quad\text{by $C$-expansivity}\\
                 &\ge C d_{X_{i_0}}(j_{i_0}(y_s), j_{i_0}(y_{s+1}))
                   \quad\text{$j_{\bullet}$ are non-expansive}.
    \end{align*}
    The effect of this is that we can always move intermediate pairs from $X_i$ back into $X_{i_0}$, at the cost of contracting distances but never by a smaller factor than $C$. In other words, \Cref{eq:psqs} dominates
    \begin{equation*}
      C d_{X_{i_0}}(p_0,q_k) = C d_{X_{i_0}}(x,x'). 
    \end{equation*}
    This, though, is precisely the $C$-expansivity claim for the map $\iota_{i_0}:X_{i_0}\to X$.
    
  \item\label{item:26} The second claim, on $C$-expansivity, follows from \Cref{eq:precdx} and the $C$-expansivity of the individual $j_i:Y\to X_i$. We thus focus on the first claim, to which end we write
    \begin{equation*}
      \delta(y,y'):=\inf_{i\in I}d_{X_i}(j_i y,\ j_i y'),\ y,y'\in Y
    \end{equation*}
    (i.e. the expression we ultimately want to equate to $d_X(\iota y,\iota y')$; cf. \Cref{eq:precdx}). It is bi-Lipschitz to the original metric $d_Y$, and hence also a complete metric on $Y$. One inequality between the two distances is obvious:
    \begin{equation}\label{eq:dxledelta}
      d_X(\iota y,\iota y')\le \delta(y,y'),\ \forall y,y'\in Y
    \end{equation}
    simply because all maps in sight are non-expansive.

    By \Cref{def:strmchn}, in a streamlined chain $(p_s,q_s)$ connecting $j_i y$ and $j_{i'} y'$, all points $p_s$ and $q_s$ belong to various images of $Y$ through $j_{\bullet}:Y\to X_{\bullet}$:
    \begin{equation*}
      p_s=j_{i_s}(y_s),\ q_s=j_{i_s}(y_{s+1}),
    \end{equation*}
    with $y_0=y$ and $y_n=y'$. It follows that \Cref{eq:psqs} dominates
    \begin{equation*}
      \delta(y,y_1)+\delta(y_1,y_2)+\cdots+\delta(y_{n-1},y')\ge \delta(y,y'),
    \end{equation*}
    and hence
    \begin{equation*}
      d_X(\iota y,\iota y')\ge \delta(y,y'),\ \forall y,y'\in Y. 
    \end{equation*}
    Having noted the opposite inequality in \Cref{eq:dxledelta}, we are done.   

  \item\label{item:28} Consider a streamlined chain $(p_s,q_s)_{s=0}^n$ connecting $x$ and $x'$. As in the proof of part \Cref{item:24}, for intermediate values $0<s<n$ (if any) we have
    \begin{equation*}
      p_s=j_{i_s}(y_s),\ q_s = j_{i_s}(y_{s+1}),\ y_{\bullet}\in Y.
    \end{equation*}
    This means that 
    \begin{equation*}
      \sum_{s=1}^{n-1} d(p_s,q_s)\ge \sum_{s=1}^{n-1} d_X(\iota y_1,\iota y_{n}),
    \end{equation*}
    so we may as well replace the original sum \Cref{eq:psqs} with a (no-larger) three-term sum as in \Cref{eq:dii}, with $y=y_1$ and $y'=y_n$.

    If there are {\it no} intermediate values $0<s<n$, i.e. the sum \Cref{eq:psqs} has at most two terms, then at least one of $x$ and $x'$ belongs to the respective image of $Y$. If that is the case for $x'$, say, then take $y'=j_{i'}^{-1}(x')$ in \Cref{eq:dii}.
    
  \item\label{item:19} This is very similar to the preceding argument, the only difference being that we also have to consider single-term sums \Cref{eq:psqs}. 
    
  \item\label{item:18} We take it for granted that Cauchy sequences with convergent subsequences are themselves convergent (\cite[proof of Lemma 43.1]{mnk}). 

    For a Cauchy sequence $(x_n)_n$ in $X$ there are two possibilities to consider: either some subsequence is contained in a single $\iota_i(X_i)\subseteq X$, or not. In the former case that subsequence is Cauchy, hence the image of a Cauchy sequence in $X_i$ by the $C$-expansivity of $\iota_i:X\to X$ (part \Cref{item:22}), hence convergent by the completeness of $X_i$.
    
    In the latter case the Cauchy property implies (via parts \Cref{item:21} and \Cref{item:27}, for instance) that we can find, for arbitrarily small $\varepsilon$, arbitrarily large $n$ with $x_n$ within $\varepsilon$ of $\iota(Y)\subseteq X$. Or: a subsequence $(x_{n_k})_k$ with
    \begin{equation*}
      d_X(x_{n_k},\iota(Y))\underset{\scriptstyle k}{\longrightarrow} 0.
    \end{equation*}
    This in turn gives a sequence $y_k\in Y$ with
    \begin{equation*}
      d_X(x_{n_k},\iota y_k)\underset{\scriptstyle k}{\longrightarrow} 0.
    \end{equation*}
    In particular $(\iota y_k)_k$ is Cauchy (because $(x_{n_k})_k$ is), and hence so is $(y_k)_k$ by the $C$-expansivity of $\iota$ (part \Cref{item:24}). $Y$ being complete, $(y_k)$, $(\iota y_k)$ and hence also $(x_{n_k})$ are all convergent. As observed, so, then, is $(x_n)_n$.
  \end{enumerate}
  This concludes the proof.
\end{proof}

A sample application:

\begin{proposition}\label{pr:gluecc}
  Let \Cref{eq:jis} be a family of $C$-expansive morphisms in \cat{CMet} for some $1\ge C>0$ and assume that
  \begin{itemize}
  \item the $X_i$ are convex in the sense of \Cref{def:cvx};
  \item the infima \Cref{eq:dii} are all achieved for arbitrary $i,i'\in I$ and $x\in X_i$, $x'\in X_{i'}$;
  \item as are the infima \Cref{eq:precdx}, for arbitrary $y,y'\in Y$.
  \end{itemize}
  The pushout $X:=\coprod_{Y,i}X_i$ is then convex. 
\end{proposition}
\begin{proof}
  This is a fairly straightforward consequence of \Cref{th:largepush}. Consider, to fix ideas, $x\in X_i$ and $x'\in X_{i'}$ for $i\ne i'$ with finite distance in $X$ (per \Cref{def:cvx}, there is nothing to check for infinite-distance pairs).

  \Cref{th:largepush} \Cref{item:21} gives us the distance $d_X(\iota_i x,\iota_{i'}x')$ via \Cref{eq:dii}. That infimum is by assumption achieved for two specific points $y,y'\in Y$, and in turn the middle term of that sum can be obtained (again by assumption) as
  \begin{equation*}
    d_X(\iota y,\iota y') = d_{X_{i_0}}(j_{i_0} y,\ j_{i_0} y')
  \end{equation*}
  for some $i_0\in I$. But now we can
  \begin{itemize}
  \item connect $x$ to $j_i y$ in $X_i$ with a minimizing geodesic (because $X_i$ is assumed convex);
  \item similarly, connect $j_{i'}y'$ to $x'$ in $X_{i'}$ with a minimizing geodesic;
  \item and also connect, once more minimally, $j_{i_0}y$ and $j_{i_0}(y')$ in $X_{i_0}$. 
  \end{itemize}
  Splicing together these three metric segments will give one such segment connecting $\iota_i x$ and $\iota_{i'}x'$ in $X$, as desired.

  The argument proceeds analogously for $i=i'$, etc. 
\end{proof}

A variant of \Cref{pr:gluecc}:

\begin{corollary}\label{cor:gluecc}
  Let \Cref{eq:jis} be a family of $C$-expansive morphisms in \cat{CMet} for some $1\ge C>0$ and assume that
  \begin{itemize}
  \item the $X_i$ are convex in the sense of \Cref{def:cvx};
  \item $Y$ is compact;
  \item and the infima \Cref{eq:precdx} are achieved for arbitrary $y,y'\in Y$.
  \end{itemize}
  The pushout $X:=\coprod_{Y,i}X_i$ is then convex. 
\end{corollary}
\begin{proof}
  The compactness of $Y$ automatically implies the second condition in the statement of \Cref{pr:gluecc}, hence the conclusion by applying that earlier result.
\end{proof}

The discussion thus far will help produce an example that shows \cat{CCMet} not to be closed, answering a question posed in \cite[discussion following Remark 5.9]{dr}.

\begin{example}\label{ex:ccmetnclosed}
  Consider a closed connected {\it Riemannian manifold} $X_0$ equipped with its intrinsic metric $d_0$ \cite[\S 7.2, Definition 2.4]{crm}, together with
  \begin{itemize}
  \item a {\it closed geodesic} $\gamma_0:[0,1]\to X_0$ representing a non-trivial element of the fundamental group. {\it Any} non-trivial free homotopy loop class contains a close geodesic, by a classical theorem of Cartan \cite[\S 12.2, Theorem 2.2]{crm}.
  \item and a point $p_0\in X_0\setminus \gamma$ in the same connected component as $\gamma_0$. 
  \end{itemize}

  We can form an abstract object $(Y,d_Y)$ of \cat{CMet} consisting of a closed length-1 geodesic and a point having infinite distance to all points of the geodesic, and embed it into $(X_0,d_0)$ in the obvious fashion: by identifying the closed geodesic with $\gamma$ isometrically and the isolated point with $p_0$.
  
  Next, we also embed $(Y,d_Y)$ into closed Riemannian manifolds $(X_{\varepsilon},d_{\varepsilon})$, so that
  \begin{itemize}
  \item the closed geodesic in $X$ is again identified isometrically with a closed length-1 geodesic $\gamma_{\varepsilon}$ in $X_{\varepsilon}$;
  \item while the isolated point of $X$ maps to a point $p_{\varepsilon}\in X_{\varepsilon}$, admitting a unique minimal-length geodesic connecting it to every in $\gamma_{\varepsilon}$, with that length equal to $L+\varepsilon$ for some large $L$:
    \begin{equation}\label{eq:lbig}
      L >
      \max\{d_0(p,\gamma(t))\ |\ t\in [0,1]\}. 
    \end{equation}
  \end{itemize}
  This can be arranged by altering the usual Riemannian metric on a sphere, with $\gamma_{\varepsilon}$ being an equator and $p_{\varepsilon}$ a pole.

  We now have the full package for gluing: embeddings $j_i:(Y,d_Y)\to (X_i,d_i)$ for $i\in \bR_{\ge 0}$. That the glued space
  \begin{equation*}
    (X,d_X)
    :=
    \coprod_Y (X_i,d_i)
  \end{equation*}
  is complete convex (i.e. an object of \cat{CCMet}) now follows from \Cref{th:largepush} \Cref{item:23} (completeness) and \Cref{cor:gluecc} (convexity): closed connected Riemannian manifolds are convex \cite[\S 7.2, Corollary 2.7]{crm}, $Y$ is compact, and the infimum \Cref{eq:precdx} is achieved by the embedding into $X_0$ by \Cref{eq:lbig}.

  The now is that the internal hom $[\bS^1,X]_{\cat{CCMet}}$ does not exist, where $\bS^1$ is the unit-length unit circle. Indeed, consider the embedding $\iota:\bS^1\to X$ identifying the circle with the closed geodesic to which all $\gamma_i\subset X_i$ collapse. That embedding can be homotoped to the constant map
  \begin{equation*}
    \mathrm{ct}:\bS^1\to p_i\in X_i\to X\quad (\text{any }i\in \bR_{\ge 0})
  \end{equation*}
  along $X_\varepsilon\subset X$, $\varepsilon>0$ by paths of respective length $L+\varepsilon$ (but no shorter), whereas in $X_0$ the closed geodesic is by assumption not homotopic to a constant map. It follows that the distance between
  \begin{equation*}
    \iota:\bS^1\to X\quad\text{and}\quad\mathrm{ct}:\bS^1\to X
  \end{equation*}
  in $[\bS^1,X]_{\cat{CMet}}$ is $L$, but that infimum is not achievable by a path of length $L$. By \Cref{cor:ihomchar}, the internal hom in \cat{CCMet} (rather than \cat{CMet}) does not exist.
\end{example}

\subsection{Limits along trees and automatic completeness}\label{subse:limtr}

One handy procedure for producing metric spaces (employed below, a number of times) is to repeatedly glue metric segments of various lengths to the same initial space. This is a (perhaps infinite) iteration of the pushout construction discussed in \Cref{subse:glue}, so results that ensure the automatic completeness of such glued spaces will be useful. Such results are the focus of the present subsection.

It will be convenient to consider colimits along {\it graphs} rather than categories; the former can be turned into the latter via the {\it free category} construction of \cite[\S 1.7]{awo} or \cite[\S II.7]{mcl}. Per those sources:

\begin{definition}
  An {\it oriented graph} (or just plain `graph', unless specified otherwise) is a quadruple
  \begin{equation*}
    (E,V,\partial_i) = (E,V,\partial_0,\partial_1)
  \end{equation*}
  consisting of
  \begin{itemize}
  \item a set $E$ of {\it edges};
  \item a set $V$ of {\it vertices};
  \item and {\it source} and {\it target} maps $\partial_0:E\to V$ and $\partial_1:E\to V$ respectively.
  \end{itemize}
  We also refer to the source and target of $e\in E$ as the {\it extremities} or {\it vertices of} $e$.
\end{definition}
The reader should picture edges as
\begin{equation*}
  \begin{tikzpicture}[auto,baseline=(current  bounding  box.center)]
    \path[anchor=base] 
    (0,0) node (l) {$\partial_1(e)$}
    +(2,0) node (r) {$\partial_0(e)$.}
    ;
    \draw[->] (r) to[bend left=0] node[pos=.5,auto,swap] {$\scriptstyle e$} (l);
  \end{tikzpicture}
\end{equation*}
The usual graph-theoretic language (paths, cycles, etc.) applies, with the caveat that much of the combinatorial literature on graphs tends to assume an edge is uniquely determined by its source and target: see e.g. \cite[\S\S I.1 and I.2]{boll-grph} or \cite[\S\S 1.1-1.5]{diest-grph}.

\begin{definition}
  Let $\Gamma:=(E,V,\partial_i)$ be a graph.
  \begin{itemize}
  \item An {\it oriented path} (or just `path') {\it of length $n$} is a sequence $(e_i)_{i=1}^n$ of edges with $\partial_1(e_i)=\partial_0(e_{i+1})$ for all $1\le i\le n-1$. Pictorially:
    \begin{equation*}
      \begin{tikzpicture}[auto,baseline=(current  bounding  box.center)]
        \path[anchor=base] 
        (0,0) node (1) {$\bullet$}
        +(2,0) node (2) {$\bullet$}
        +(3,0) node (3) {$\cdots$}
        +(4,0) node (4) {$\bullet$}
        +(6,0) node (5) {$\bullet$}
        +(8,0) node (6) {$\bullet$}
        ;
        \draw[->] (6) to[bend left=0] node[pos=.5,auto,swap] {$\scriptstyle e_1$} (5);
        \draw[->] (5) to[bend left=0] node[pos=.5,auto,swap] {$\scriptstyle e_2$} (4);
        \draw[->] (2) to[bend left=0] node[pos=.5,auto,swap] {$\scriptstyle e_n$} (1);
      \end{tikzpicture}
    \end{equation*}
    We will also occasionally refer to the {\it empty path} based at a vertex $v\in V$, consisting of no edges at all (and hence of length 0). 
  \item An {\it oriented cycle} (or just `cycle') is a path $(e_i)_{i=1}^n$ such that $\partial_1(e_n)=\partial_0(e_1)$ (i.e. the target of the last arrow is the source of the first; the path returns to its origin).
  \item {\it Unoriented} paths and cycles are defined as their oriented cousins, except that for each $i$ only the weaker requirement
    \begin{equation*}
      \partial_1(e_i)=\partial_0(e_{i+1})
      \quad\text{or}\quad
      \partial_1(e_{i+1})=\partial_0(e_{i})
    \end{equation*}
    is made:
    \begin{equation*}
      \begin{tikzpicture}[auto,baseline=(current  bounding  box.center)]
        \path[anchor=base] 
        (0,0) node (1) {$\cdots$}
        +(1,0) node (2) {$\bullet$}
        +(3,0) node (3) {$\bullet$}
        +(5,0) node (4) {$\bullet$}
        +(7,0) node (5) {$\bullet$}        
        ;        
        \draw[->] (5) to[bend left=0] node[pos=.5,auto,swap] {$\scriptstyle e_1$} (4);
        \draw[<-] (4) to[bend left=0] node[pos=.5,auto,swap] {$\scriptstyle e_2$} (3);
        \draw[->] (3) to[bend left=0] node[pos=.5,auto,swap] {$\scriptstyle e_3$} (2);
      \end{tikzpicture}
    \end{equation*}
    is an unoriented path, for instance.
  \item $\Gamma$ is {\it connected} if every two vertices are extremities of edges belonging to a common unoriented path.
  \item $\Gamma$ is a {\it forest} if it has no unoriented cycles;
  \item and a {\it tree} if it is in addition connected.
  \end{itemize}
\end{definition}

Recall from \cite[\S 1.7]{awo} (or \cite[\S II.7, Theorem 1]{mcl}):

\begin{definition}
  The {\it free category} $\cC(\Gamma)$ on a graph $\Gamma=(E,V,\partial_i)$ has
  \begin{itemize}
  \item $V$ as its set of vertices;
  \item the (oriented) paths of $\Gamma$ as morphisms;
  \item the empty paths as identity morphisms;
  \item and concatenation of paths as composition.
  \end{itemize}
  Via this construction, we transport terminology involving categories to graphs. Thus, {\it $\Gamma$-functors} are functors defined on $\cC(\Gamma)$, {\it $\Gamma$-colimits} are colimits of such functors, etc.
\end{definition}

In order to state \Cref{th:tree} recall also (\cite[discussion preceding Proposition 1.3.2]{diest-grph})

\begin{definition}\label{def:diam}  
  The {\it distance} between two vertices $x,y\in V$ of a graph $(E,V,\partial_i)$ is the length of a shortest unoriented path containing edges having $x$ and $y$ as endpoints; it is zero if $x=y$ and infinite if no such paths exist.

  The {\it diameter} of a graph is the supremum of the distances between pairs of vertices.
\end{definition}

\begin{remark}
  It is easy to see that for a tree the diameter is also the supremum of the lengths of unoriented paths; this is implicit in the proof of \Cref{th:tree}.
\end{remark}

\begin{theorem}\label{th:tree}
  Let $1\ge C>0$, $\Gamma=(E,V,\partial_i)$ a finite-diameter forest, and
  \begin{equation*}
    V\ni v\xmapsto{\quad F\quad} (X_v,d_v)\in \cat{CMet}
  \end{equation*}
  a $\Gamma$-functor.
  \begin{enumerate}[(a)]
  \item\label{item:29} There is a non-negative integer $N\in \bZ_{>0}$, depending only on the diameter $\mathrm{diam}(\Gamma)$, such that the canonical morphisms from $X_v$ to the colimit
    \begin{equation*}
      (X,d):=\varinjlim F\in \cat{Met}
    \end{equation*}
    in \cat{Met} is $C^N$-expansive.
  \item\label{item:30} That colimit is automatically complete, and hence also a colimit in \cat{CMet}.
  \end{enumerate}
\end{theorem}
\begin{proof}
  We proceed by induction on the diameter $\mathrm{diam}(\Gamma)$. When it is zero the graph consists of only isolated vertices and the colimit is a coproduct; there is thus nothing to prove ($N=0$ will do in that case). A diameter-1 forest is a disjoint union of edges $e$, so a functor thereon is a collection of contractions
  \begin{equation*}
    (X_e,d_{X,e})\longleftarrow (Y_e,d_{Y,e}).
  \end{equation*}
  The colimit is the coproduct of the $X_e$; we are once more done, taking $N=1$.

  We now take for granted the claim for smaller diameters, assuming $\mathrm{diam}(\Gamma)\ge 2$. It is also harmless to work with a {\it tree} $\Gamma$, since a colimit over a forest will be a coproduct of colimits over its connected components (i.e. constituent trees). 

  \begin{enumerate}[(1)]

  \item {\bf Pruning leaf sources.} A {\it leaf} \cite[\S 1.5]{diest-grph} is a vertex attached to a single edge, and a {\it source} is a vertex that is not the target of any edges. Consider the tree $\Gamma'\subseteq \Gamma$ obtained by deleting the respective edges $e_v$ adjacent to leaf sources $v$:
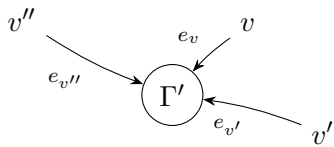
    \begin{equation*}
      \begin{tikzpicture}[>=Stealth,auto,baseline=(current  bounding  box.center)]
        \path[anchor=base] 
        (0,0) node[circle,draw] (g) {$\Gamma'$}
        +(1,1) node (v) {$v$}
        +(2,-.5) node (v') {$v'$}
        +(-2,1) node (v'') {$v''$}
        ;
        \draw[<-] (g) to[bend left=6] node[pos=.5,auto] {$\scriptstyle e_v$} (v);
        \draw[<-] (g) to[bend left=6] node[pos=.5,auto,swap] {$\scriptstyle e_{v'}$} (v');
        \draw[<-] (g) to[bend left=6] node[pos=.5,auto] {$\scriptstyle e_{v''}$} (v'');
      \end{tikzpicture}
    \end{equation*}
    We have an isomorphism
    \begin{equation*}
      \varinjlim(F|_{\Gamma'})\cong \varinjlim F,
    \end{equation*}
    so we can always (as suggested by the language used above) prune such edges $e_v$. Should the deletion of $e_v$ expose a new leaf source
    \begin{equation*}
      \bullet\xleftarrow{\quad e_w\quad} w
    \end{equation*}
    (that used to be the target of the now-absent $e_v$), $e_v$ and $e_w$ are composable to a path. This observation replicates, and since there is a bound of $\mathrm{diam}(\Gamma)$ on the length of an unoriented path the diameter also caps the number of times this pruning procedure can be repeated.

    In conclusion, we can assume that all leaves are {\it sinks} (i.e. targets only) without affecting the desired conclusions.
    
  \item {\bf All leaves are sinks.} Because every path of maximal length must contain (an edge adjacent to) a leaf, $\Gamma$ consists of a tree $\Gamma'$ of strictly smaller diameter connected to the leaves $v$ via their unique respective incident edges $e_v$:
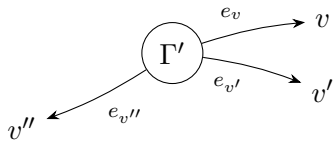
    \begin{equation*}
      \begin{tikzpicture}[>=Stealth,auto,baseline=(current  bounding  box.center)]
        \path[anchor=base] 
        (0,0) node[circle,draw] (g) {$\Gamma'$}
        +(2,.5) node (v) {$v$}
        +(2,-.5) node (v') {$v'$}
        +(-2,-1) node (v'') {$v''$}
        ;
        \draw[->] (g) to[bend left=6] node[pos=.5,auto] {$\scriptstyle e_v$} (v);
        \draw[->] (g) to[bend left=6] node[pos=.5,auto,swap] {$\scriptstyle e_{v'}$} (v');
        \draw[->] (g) to[bend left=6] node[pos=.5,auto] {$\scriptstyle e_{v''}$} (v'');
      \end{tikzpicture}
    \end{equation*}
    We have the induction hypothesis for $\Gamma'$, and $(X,d)$ can be obtained as the pushout of the morphisms
    \begin{equation*}
      \varinjlim\left(F|_{\Gamma'}\right)\xrightarrow{\quad F(e_v) \quad} (X_v,d_v),\quad\text{leaves }v.
    \end{equation*}
    The conclusion follows from \Cref{th:largepush}. 
  \end{enumerate}
  This concludes the argument. 
\end{proof}

\Cref{th:tree} can presumably be generalized in a number of ways, but not too cavalierly.

\begin{examples}
  \begin{enumerate}[(1)]
  \item\label{item:31} Dropping the finite-diameter condition, even for tree colimits of isometries (i.e. $C=1$ in \Cref{th:tree}). Consider the graph
    
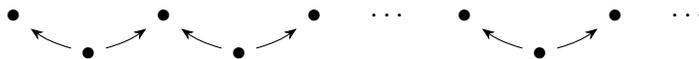
\begin{equation*}
      \begin{tikzpicture}[>=Stealth,auto,baseline=(current  bounding  box.center)]
        \path[anchor=base] 
        (0,0) node (1) {$\bullet$}
        +(1,-.5) node (d1) {$\bullet$}
        +(2,0) node (2) {$\bullet$}
        +(3,-.5) node (d2) {$\bullet$}
        +(4,0) node (3) {$\bullet$}
        +(5,0) node () {$\cdots$}
        +(6,0) node (4) {$\bullet$}
        +(7,-.5) node (d3) {$\bullet$}
        +(8,0) node (5) {$\bullet$}
        +(9,0) node () {$\cdots$}
        ;
        \draw[->] (d1) to[bend left=10] node[pos=.5,auto] {$\scriptstyle $} (1);
        \draw[->] (d1) to[bend right=10] node[pos=.5,auto] {$\scriptstyle $} (2);
        \draw[->] (d2) to[bend left=10] node[pos=.5,auto] {$\scriptstyle $} (2);
        \draw[->] (d2) to[bend right=10] node[pos=.5,auto] {$\scriptstyle $} (3);
        \draw[->] (d3) to[bend left=10] node[pos=.5,auto] {$\scriptstyle $} (4);
        \draw[->] (d3) to[bend right=10] node[pos=.5,auto] {$\scriptstyle $} (5);
      \end{tikzpicture}
    \end{equation*}
    The top row is assigned closed segments of respective lengths $\frac 1{2^n}$, $n\ge 1$ decreasing rightward, the bottom row is assigned points (i.e. singletons, regarded as objects of \cat{CMet}), and the arrows are identifications with endpoints, splicing together the segments.

    The colimit in \cat{Met} is then a length-1 half-open segment, hence not complete.

  \item\label{item:32} Even finite graphs will not do, for either of the claims in \Cref{th:tree}, if loops are allowed (even if $C=1$, i.e. the connecting morphisms are isometries). Consider for instance a diagram in \cat{CMet} of the form
    \begin{equation}\label{eq:loop}
      \begin{tikzpicture}[auto,baseline=(current  bounding  box.center)]
        \path[anchor=base] 
        (0,0) node (1) {$(X,d)$}
        ;
        \draw[->] (1) .. controls (2,-1) and (2,1) ..
        node[pos=.5,auto,swap] {$\scriptstyle \varphi$} (1);
      \end{tikzpicture}
    \end{equation}
    as we now describe.
    \begin{itemize}
    \item $(X,d)$ is the portion of the first quadrant above a hyperbola:
      \begin{equation*}
        (X,d):=\{(x,y)\in \bR_{\ge 0}^2\ |\ xy\ge 1\},
      \end{equation*}
      with the usual Euclidean distance.
    \item the isometry $\varphi$ shifts everything up by 1:
      \begin{equation*}
        \varphi(x,y) := (x,y+1). 
      \end{equation*}
    \end{itemize}
    The colimit in \cat{Met} is the space of orbits
    \begin{equation}\label{eq:phiorb}
      O_x:=\{\varphi^n x\ |\ n\in \bZ_{\ge 0}\},\ x\in X
    \end{equation}
    of $\varphi$ with the usual set-distance induced by $d$:
    \begin{equation*}
      d(O_x,O_y) = \inf\{d(p,q)\ |\ p\in O_x,\ q\in O_y\}.
    \end{equation*}
    The (images of the) points $\left(\frac 1{2^n},\ 2^n\right)$, $n\in \bZ_{\ge 0}$ form a Cauchy sequence: the successive distances between their $\varphi$-orbits are, respectively, $\frac 1{2^n}$.

    Clearly though, that sequence has no limit in the quotient.    
        
  \item\label{item:33} The same effect can easily be replicated with finite graphs without single-edge loops:
    \begin{equation*}
      \begin{tikzpicture}[auto,baseline=(current  bounding  box.center)]
        \path[anchor=base] 
        (0,0) node (l) {$X$}
        +(2,0) node (r) {$X$}
        ;
        \draw[->] (l) to[bend left=16] node[pos=.5,auto] {$\scriptstyle \varphi$} (r);
        \draw[->] (r) to[bend left=16] node[pos=.5,auto] {$\scriptstyle \id$} (l);        
      \end{tikzpicture}
    \end{equation*}
    or
    \begin{equation*}
      \begin{tikzpicture}[auto,baseline=(current  bounding  box.center)]
        \path[anchor=base] 
        (0,0) node (l) {$X$}
        +(2,.5) node (u) {$X$}
        +(2,-.5) node (d) {$X$}
        +(4,0) node (r) {$X$}
        ;
        \draw[->] (l) to[bend left=6] node[pos=.5,auto] {$\scriptstyle \id$} (u);
        \draw[->] (u) to[bend left=6] node[pos=.5,auto] {$\scriptstyle \varphi$} (r);
        \draw[->] (l) to[bend right=6] node[pos=.5,auto,swap] {$\scriptstyle \id$} (d);
        \draw[->] (d) to[bend right=6] node[pos=.5,auto,swap] {$\scriptstyle \id$} (r);
      \end{tikzpicture}
    \end{equation*}
    say.

  \item\label{item:34} On the other hand, it was crucial that the isometry $\varphi$ of \Cref{eq:loop} not be bijective. In other words, the colimit in \cat{Met} of a diagram \Cref{eq:loop} with a metric {\it isomorphism} $\varphi$ will automatically be complete if the original space $(X,d)$ is.

    This is follows from the simple remark recorded as \Cref{le:groupoid} below, given that for an isomorphism $\varphi$ the diagram \Cref{eq:loop} can be regarded as a functor defined on a groupoid.
    
  \end{enumerate}
\end{examples}

\begin{lemma}\label{le:groupoid}
  A colimit in $\cat{Met}$ of a functor
  \begin{equation*}
    F:\cG\to \cat{CMet}
  \end{equation*}
  defined on a groupoid is automatically complete, and hence also a colimit in \cat{CMet}.
\end{lemma}
\begin{proof}
  Since the colimit in question is the coproduct of the restrictions $F|_{\cG_i}$ to the connected components of $\cG$, it is enough to assume $\cG$ is connected to begin with. But then it will be isomorphic to a group $\Gamma$, so the colimit in $\cat{Met}$ is the coequalizer of the isometries constituting a group action
  \begin{equation*}
    \Gamma\times (X,d)\to (X,d). 
  \end{equation*}
  Such a coequalizer is obtained by identifying the orbits \Cref{eq:phiorb} to single points and further identifying those that are distance zero apart. If $(O_{x_n})_n$ is a Cauchy sequence in that quotient then we can assume, perhaps after passing to a subsequence, that
    \begin{equation*}
      d(O_{x_{n}},\ O_{x_{n+1}}) < \frac 1{2^n},\ \forall n\ge 1. 
    \end{equation*}
    Having fixed $x_1$, {\it some} $x\in O_{x_2}$ is less than $\frac 12$ from it by assumption, and upon translating $x$ by some $g\in \Gamma$ we may as well assume that $x=x_2$. Similarly, we can assume $d(x_2,x_3)<\frac 1{2^2}$, etc.

    The limit of $(x_n)_n$ in the complete space $(X,d)$ will map to a limit of $(O_{x_n})_n$.
\end{proof}

\begin{remark}
  \Cref{le:groupoid} only gives an analogue of part \Cref{item:30} of \Cref{th:tree}; the corresponding version of \Cref{item:29} does {\it not} hold in general, of course (since the identification of a $\Gamma$-orbit to a single point decreases positive distances to zero).
\end{remark}

\subsection{The convex pseudo-reflection}\label{subse:prfl}

The gluing results above make it very easy to produce convex metric spaces by simply attaching the possibly-missing metric segments.

\begin{definition}\label{def:cvxcompl}
  Let $(X,d)$ be a metric space.
  \begin{itemize}
  \item For points $x,x'\in X$ with $d(x,x')<\infty$ the {\it $(x,x')$-convex completion} $\overline{X}^{x,x'}$ is the space obtained by gluing a length-$d(x,x')$ metric segment with endpoints $x$ and $x'$ to $X$. 
  \item More generally, for a set $\cS$ of finite-distance point-pairs in $X$ the {\it $\cS$-convex completion} $\overline{X}^{\cS}$ is the space obtained by gluing a metric segment as above, for each pair in $\cS$.    
  \item Finally, the {\it convex completion} $\overline{X}^{\cat{cvx}}$ is obtained by gluing one metric segment connecting {\it every} finite-distance pair of distinct points.
  \end{itemize}
\end{definition}

The construction $X\mapsto \overline{X}^{\cat{cvx}}$ is a version of the {\it weak reflection} described in \cite[Remark 6.9]{dr} (hence the title of the present subsection), with some differences. The following result, for instance, makes it clear that in order to produce a convex metric space one need not iterate the construction recursively.

Not only is $\overline{X}^{\cat{cvx}}$, as the name suggests, convex, but we can forego the redundant gluing. 

\begin{proposition}\label{pr:cvxcmpl}
  Let $(X,d)$ be a complete metric space and $\cS\subset X^2$ a set of finite-distance pairs of points containing all pairs which are {\it not} connected in $X$ by a metric segment.

  The $\cS$-convex completion $\overline{X}^{\cS}$ of \Cref{def:cvxcompl} is complete and convex.
\end{proposition}
\begin{proof}
  Note first that every metric-segment gluing is a pushout (ordinary, not multiple) for a pair of arrows
  \begin{equation*}
    {\bf 2}_{\delta}\to \gamma
    \quad\text{and}\quad
    {\bf 2}_{\delta}\to X:
  \end{equation*}
  the identification with the endpoints of a segment $\gamma$ of length $\delta:=d(x,x')$ and with the two points $x,x'\in X$. These embeddings are both isometries, so we can apply the above discussion on $C$-expansive maps with $C=1$. This implies in particular, by \Cref{th:largepush} \Cref{item:22}, that in every such gluing the original space (to which the segment is being glued) embeds isometrically into the glued space; we use this implicitly and repeatedly.
  
  By construction, $\overline{X}^{\cS}$ is the directed union of the partially-``convexified'' $\overline{X}^{\cF}$ for finite sets $\cF\subseteq \cS$ of point-pairs in $X$. We now turn to the claims.
  
  {\bf (1): Completeness.} This is a consequence of \Cref{th:tree} \Cref{item:30}: $\overline{X}^{\cS}$ is the colimit in \cat{Met} along the tree obtained by gluing the various sub-trees
  \begin{equation*}
    \begin{tikzpicture}[auto,baseline=(current  bounding  box.center)]
      \path[anchor=base] 
      (0,0) node (l) {$X$}
      +(2,-.5) node (d) {${\bf 2}_{\delta}$}
      +(4,0) node (r) {$\gamma$}
      ;
      \draw[->] (d) to[bend left=6] node[pos=.5,auto] {$\scriptstyle $} (l);
      \draw[->] (d) to[bend right=6] node[pos=.5,auto] {$\scriptstyle $} (r);      
    \end{tikzpicture}
  \end{equation*}
  along the common vertex $X$. This produces a tree of diameter $\le 4$, so \Cref{th:tree} applies with $C=1$.
 
  {\bf (2): Convexity.} As noted, $X$ itself embeds into $\overline{X}^{\cS}$ isometrically. It follows that, by construction, any two points therein are (if a finite distance apart) the endpoints of a metric segment.

  Points on the same glued metric segment $\gamma$ are connected by a portion of $\gamma$ itself (which embeds isometrically into $\overline{X}^{\cS}$ by \Cref{th:largepush} \Cref{item:22} with $C=1$). 
  
  On the other hand, for a point $x\in X$ and one $y\in \gamma$ on one of the glued segments, $d(x,y)<\infty$ implies that the distances from $x$ to the endpoints $p$ and $q$ of $\gamma$ are also finite. Furthermore, \Cref{th:largepush} \Cref{item:21} shows that
  \begin{equation*}
    d(x,y) = \min\left(d(x,p)+d(p,y),\ d(x,q)+d(q,y)\right).
  \end{equation*}
  In either case we have the desired metric segments in $\overline{X}^{\cS}$: for $p$, say, there is one connecting $x$ to $p$ (either originally in $X$ or attached upon constructing $\overline{X}^{\cS}$) and one connecting $p$ to $y\in \gamma$ (a portion of $\gamma$ itself).

  Finally, for points on distinct glued segments $\gamma\ne \gamma'$ we can fall back on the preceding argument by first enlarging $X$ with the addition of one segment, and then further gluing the other.
\end{proof}

In particular, taking for $\cS$ the set of all finite-distance point-pairs, we have

\begin{corollary}\label{cor:cvxcmpl}
  For any complete metric space $(X,d)$ the convex completion $\overline{X}^{\cat{cvx}}$ of \Cref{def:cvxcompl} is convex and complete.  \qedhere
\end{corollary}

\section{Intrinsic spaces and the path comonad}\label{se:pcomon}

\Cref{ex:ccmetnclosed} highlights one reason why \cat{CCMet} is inadequate as an enriching category: lack of closure. As seen in \Cref{re:cvxchar}, convexity essentially means being strictly intrinsic; if we compromise on the stricture, coreflections are much easier to come by. Denoting by $\cat{CPMet}$ the category of complete path metric spaces (in the sense of \Cref{def:lspace}), we have the following version of \Cref{pr:whencorefl};

\begin{proposition}\label{pr:lengthcorefl}
  For a complete metric space $(X,d_X)\in \cat{CMet}$ the associated length metric space $(X,d_{X,\ell})$, equipped with the identity contraction
  \begin{equation}\label{eq:coreflp}
    \id: (X,d_{X,\ell})\to (X,d_{X}),
  \end{equation}
  is a coreflection of $(X,d_X)$ in \cat{CPMet}.
\end{proposition}
\begin{proof}
  The completeness of the metric $d_{X,\ell}$ is part of \Cref{pr:whencorefl}, the fact that it is indeed a length metric was observed in \Cref{res:dl} \Cref{item:2}, and the universality property of \Cref{eq:coreflp} is not harder to prove than the analogous claim in \Cref{pr:whencorefl}.
\end{proof}

This affords an analogue of \Cref{cor:ihomchar}. Before stating it, we observe that the setup of \Cref{le:corefl} obtains for \cat{CPMet}, just as it did for \cat{CCMet}.

\begin{lemma}\label{le:cpmetmon}
  The full subcategory
  \begin{equation*}
    \cat{CPMet}\subset \cat{CMet}
  \end{equation*}
  of complete path metric spaces contains the monoidal unit and is closed under tensor products, so is a full monoidal subcategory.
\end{lemma}
\begin{proof}
  The monoidal unit is the one-point space, whose metric is obviously intrinsic. As for closure under tensor products, consider complete path metric spaces $(X,d_X)$ and $(Y,d_Y)$ and points
  \begin{equation*}
    (x,y),\ (x',y')\in X\times Y = X\otimes Y
  \end{equation*}
  with
  \begin{equation*}
    \ell
    :=
    d_{X\otimes Y}((x,y),(x',y'))
    =
    d_X(x,x') + d_Y(y,y')
  \end{equation*}
  finite. We can concatenate
  \begin{itemize}
  \item a path
    \begin{equation*}
      [a,b]\to X\cong X\times \{y\},\ a\mapsto (x,y),\ b\mapsto (x',y)
    \end{equation*}
    of approximate length $d_X(x,x')$;
  \item and a path
    \begin{equation*}
      [b,c]\to Y\cong \{x'\}\times Y,\ b\mapsto (x',y),\ b\mapsto (x',y')
    \end{equation*}
    of approximate length $d_Y(y,y')$
  \end{itemize}
  to obtain a path in $X\otimes Y$ of length close to $\ell$. In short: $(X\otimes Y,\ d_{X\otimes Y})$ is a path metric space.
\end{proof}

Finally, the \cat{CPMet}-specific version of \Cref{cor:ihomchar}:

\begin{corollary}\label{cor:cpclosed}
  The monoidal category $\cV:=\cat{CPMet}$ of complete path metric spaces is closed: for objects $(X,d_X)$ and $(Y,d_Y)$ in $\cV$ we have
  \begin{equation*}
    [X,Y]_{\cV} = (\cat{CMet}(X,Y),d_{\sup,\ell}),
  \end{equation*}
  the intrinsic space attached to the internal hom \Cref{eq:v0int}.
\end{corollary}
\begin{proof}
  A consequence of \Cref{le:corefl} (which applies by \Cref{le:cpmetmon}) and \Cref{pr:lengthcorefl}.
\end{proof}

The coreflection
\begin{equation*}
  \cat{CMet}\to \cat{CPMet}
\end{equation*}
of \Cref{pr:lengthcorefl} allows us to make sense of the \cat{CPMet}-valued ``internal hom'' $[Y,X]_{\cat{CPMet}}$ even for arbitrary $X$ and $Y$ in \cat{CMet} (not \cat{CPMet}!): simply coreflect the original internal hom $[Y,X]_{\cat{CMet}}$. We will use this notation without further comment in the sequel.

\Cref{th:cpcomon} summarizes the ways in which \cat{CPMet} is better behaved than \cat{CCMet} as a subcategory of \cat{CMet}. Recall, briefly, the following notion (dual to that of a {\it monadic} or {\it tripleable} functor \cite[\S VI.3]{mcl}):

\begin{definition}\label{def:comon}
  Let $\cC$ be a category.
  \begin{itemize}
  \item A {\it comonad} on $\cC$ is an endofunctor $T:\cC\to \cC$ equipped with natural transformations
    \begin{equation*}
      \varepsilon:T\to \id,\quad \Delta:T\to T^2,
    \end{equation*}
    with $\Delta$ {\it coassociative} and {\it counital with respect to $\varepsilon$} in the sense that
    \begin{equation*}
      (\Delta\cdot\id)\Delta = (\id\cdot\Delta)\Delta:T\to T^3
    \end{equation*}
    and
    \begin{equation*}
      (\varepsilon\cdot\id)\Delta = \id_T = (\id\cdot\varepsilon)\Delta:T\to T
    \end{equation*}
    respectively.
  \item A {\it coalgebra} over a comonad $T$ consists of an object $c$ and a morphism $\rho:c\to Tc$ such that
    \begin{equation*}
      \begin{tikzpicture}[auto,baseline=(current  bounding  box.center)]
        \path[anchor=base] 
        (0,0) node (l) {$c$}
        +(2,.5) node (u) {$Tc$}
        +(2,-.5) node (d) {$Tc$}
        +(4,0) node (r) {$T^2c$}
        ;
        \draw[->] (l) to[bend left=6] node[pos=.5,auto] {$\scriptstyle \rho$} (u);
        \draw[->] (u) to[bend left=6] node[pos=.5,auto] {$\scriptstyle \Delta_c$} (r);
        \draw[->] (l) to[bend right=6] node[pos=.5,auto,swap] {$\scriptstyle \rho$} (d);
        \draw[->] (d) to[bend right=6] node[pos=.5,auto,swap] {$\scriptstyle T\rho$} (r);
      \end{tikzpicture}
    \end{equation*}
    
    and
    \begin{equation*}
      \begin{tikzpicture}[auto,baseline=(current  bounding  box.center)]
        \path[anchor=base] 
        (0,0) node (l) {$c$}
        +(2,.5) node (u) {$Tc$}
        +(4,0) node (r) {$c$}
        ;
        \draw[->] (l) to[bend left=6] node[pos=.5,auto] {$\scriptstyle \rho$} (u);
        \draw[->] (u) to[bend left=6] node[pos=.5,auto] {$\scriptstyle \varepsilon$} (r);
        \draw[->] (l) to[bend right=6] node[pos=.5,auto,swap] {$\scriptstyle \id$} (r);        
      \end{tikzpicture}
    \end{equation*}
    commute.

    We write $\cat{Coalg}(T)$ for the category of $T$-coalgebras.
  \item A functor $F:\cD\to \cC$ is {\it comonadic} if it is isomorphic to the forgetful functor
    \begin{equation*}
      \cat{Coalg}(T)\to \cC
    \end{equation*}
    for some comonad $T$ on $\cC$. In then follows that $F$ has a right adjoint $G$, and $T\cong FG$ regarded as a comonad by equipping it with
    \begin{itemize}
    \item $\varepsilon:FG\to \id$;
    \item and the natural transformation
      \begin{equation*}
        FG\xrightarrow[]{\quad F \eta G} FGFG,
      \end{equation*}
      where $\varepsilon$ and $\eta:\id\to GF$ are the {\it counit} and {\it unit} of the adjunction respectively \cite[Definition 19.3]{ahs}.
    \end{itemize}
  \end{itemize}
\end{definition}

\begin{theorem}\label{th:cpcomon}
  The full subcategory $\iota:\cat{CPMet}\subset \cat{CMet}$ is closed under colimits, and in particular cocomplete.

  Furthermore, the inclusion functor $\iota$ is comonadic.
\end{theorem}
\begin{proof}
  Closure under coproducts is obvious, and we have just remarked we also have closure under gluing and hence pushouts; this suffices by (the dual of) \cite[Theorem 12.3]{ahs}.

  We know from \Cref{pr:lengthcorefl} that $\iota$ is a left adjoint, and it reflects isomorphisms because it is full. Beck's theorem (dual to the \cite[\S 3.3, Theorem 10]{bw}, for instance) then shows that $\iota$ is comonadic provided
  \begin{itemize}
  \item $\cat{CPMet}$ has equalizers for those {\it reflexive} (\cite[\S 3.3, preceding Proposition 4]{bw}) parallel pairs $f,g:Y\to X$ of morphisms in \cat{CPMet} for which $\iota f$ and $\iota g$ have a {\it contractible equalizer} (\cite[\S 3.3, preceding Proposition 2]{bw}) in \cat{CMet};
  \item and $\iota$ preserves those equalizers.
  \end{itemize}
  Recall from \cite[\S 3.3, p.104, discussion preceding Proposition 2]{bw} that the aforementioned contractibility means, in particular, that in the \cat{CMet}-equalizer diagram
  \begin{equation}\label{eq:cpeqlz}
    \begin{tikzpicture}[auto,baseline=(current  bounding  box.center)]
      \path[anchor=base] 
      (0,0) node (l) {$Z$}
      +(2,0) node (u) {$Y$}
      +(4,0) node (r) {$X$}
      ;
      \draw[->] (l) to[bend left=0] node[pos=.5,auto] {$\scriptstyle h$} (u);
      \draw[->] (u) to[bend left=10] node[pos=.5,auto] {$\scriptstyle \iota f$} (r);
      \draw[->] (u) to[bend right=10] node[pos=.5,auto,swap] {$\scriptstyle \iota g$} (r);      
    \end{tikzpicture}
  \end{equation}
  (which coincides with the \cat{Set}-equalizer of $\iota f$ and $\iota g$) the inclusion $h:Z\subseteq X$ {\it splits}, i.e. has a left inverse $\pi:Y\to Z$, with $\pi h = \id_Z$.

  This makes $Z$ a {\it retract} \cite[\S 35, Exercise 4]{mnk} of $Y$. That retracts in \cat{CMet} of intrinsic metric spaces are again intrinsic is a simple exercise, hence the conclusion that \Cref{eq:cpeqlz} is in fact an equalizer in \cat{CPMet}.
\end{proof}

Furthermore, we know from \cite[Example 2.3 (2)]{ar-ap} that \cat{CMet} is locally presentable; it turns out that so is \cat{CPMet}.

\begin{theorem}\label{th:cppres}
  The category \cat{CPMet} of complete intrinsic metric spaces is locally $\aleph_1$-presentable.
\end{theorem}
\begin{proof}
  We already know from \Cref{th:cpcomon} that $\cC:=\cat{CPMet}$ is cocomplete; by \cite[Theorem 1.20]{ar}, it remains to show that it has a {\it strong generator} consisting of $\aleph_1$-presentable objects. This means \cite[\S\S 0.5 and 0.6]{ar}:
  \begin{itemize}
  \item a set $\cS$ of $\aleph_1$-presentable objects;
  \item so that every object $X\in \cC$ admits an {\it extremal epimorphism}
    \begin{equation}\label{eq:stox}
      e:\coprod S\to X
    \end{equation}
    from a coproduct of objects $S\in \cS$;
  \item in the sense that $e$ is epic and for any factorization
    \begin{equation}\label{eq:efact}
      e=m\circ-
    \end{equation}
    with monic $m$, the latter is an isomorphism.
  \end{itemize}
  The generator $\cS$ consists of the finite segments $[0,\ell]$ for $\ell\in \bR_{\ge 0}$. Clearly, these spaces are $\aleph_1$-presentable (as is every complete {\it separable} metric space). Moreover, every intrinsic metric space is a quotient of the disjoint union (i.e. coproduct) of its rectifiable curves, so we indeed have an epimorphism \Cref{eq:stox} (canonical, since we are surjecting from the disjoint union of {\it all} finite-length paths).

  It remains to argue that \Cref{eq:stox} is extremal, for which purpose we fix a factorization \Cref{eq:efact} through a monomorphism $m$. We will show in \Cref{le:monoinj} that monomorphisms are injective. Assuming this for now, $m$ will be {\it bi}jective: surjectivity is immediate from \Cref{eq:efact} and the fact that $e$ is itself onto.

  Naturally, $m$ is also contractive, as are all maps in sight. On the other hand though, $m$ cannot {\it strictly} decrease any distances: points $x,x'\in X$ a finite distance $\ell$ apart are connected by paths of lengths
  \begin{equation*}
    \ell+\varepsilon,\ \text{arbitrarily small }\varepsilon>0,
  \end{equation*}
  so their preimages through $m$ are at most $\ell+\varepsilon$ apart no matter how small $\varepsilon>0$ is. It follows that $m$ is an isometry onto $X$, i.e. an isomorphism.
\end{proof}

\begin{remark}
  It mattered, in the proof of \Cref{th:cppres}, that the set $\cS$ consisted of all segments of arbitrary finite lengths. Had we chosen a smaller $\cS$, consisting, say, of only the singleton, the argument would have fallen through: every metric space $(X,d)$ (intrinsic or not) admits a non-expansive bijection
  \begin{equation*}
    \left(\text{discrete }X\right)\to (X,d)
  \end{equation*}
  from its own discrete version, with infinite distances between distinct points. Plainly, that morphism is epic but not extremally so. Having a rich supply of segments in $\cS$ allowed the last part of the argument (wherein we connected points with ``almost-metric-segments'') to go through.
\end{remark}

\begin{lemma}\label{le:monoinj}
  In both \cat{CMet} and \cat{CPMet} the monomorphisms are the injections.
\end{lemma}
\begin{proof}
  Consider a morphism $m:Y\to X$ in either category. Being monic is equivalent to saying that the two Cartesian projections
  \begin{equation*}
    Y\times_X Y\to Y
  \end{equation*}
  are isomorphisms. This implies that $m$ is one-to-one, since $Y\times_X Y$ is nothing but the set-theoretic pullback equipped with
  \begin{itemize}
  \item the supremum norm
    \begin{equation*}
      d((y_0,y_1),(y'_0,y'_1)) = \max_{i=0,1}d_Y(y_i,y'_i)
    \end{equation*}
    in \cat{CMet}
  \item and a possibly larger norm upon coreflecting to \cat{CPMet}.
  \end{itemize}
\end{proof}

\section{Complements and asides}\label{se:aside}

\subsection{Tensors over \cat{CMet}}\label{subse:tens}

We write $\cV:=\cat{CMet}$ for brevity. One natural question, given this $\cV$-enrichment, is whether the various categories mentioned above are {\it tensored} over $\cV$ in the sense of \cite[\S 3.7]{kly}: whether, in other words, for each $X\in \cV$ and $C\in \cC$ (the category of interest) the functor
\begin{equation}\label{eq:hvhc}
  \mathrm{hom}_{\cV}(X,\mathrm{hom}_{\cC}(C,-)):\cC\to \cV
\end{equation}
is representable (note that we are regarding hom spaces in both $\cV$ and $\cC$ as objects in $\cV$; so enrichment is used repeatedly to make sense of the concept). If that is the case, we denote the representing object by $X\otimes C$. The aim here is to note that this cannot be the case for $\cC^*_1$; in fact, a fairly strong negation of tensor-existence holds:

\begin{theorem}\label{th:notensfull}
  For a unital $C^*$-algebra $C\in \cC^*_1$ and a complete metric space $(X,d)\in \cV:=\cat{CMet}$ the tensor product $X\otimes C$ over $\cV$ exists if and only if one of the following conditions holds:
  \begin{itemize}
  \item $X$ is empty, in which case $X\otimes C=\{0\}$;
  \item $X$ is a singleton, whence $X\otimes C\cong C$;
  \item or $C$ is at most 1-dimensional (i.e. $\{0\}$ or $\bC$), so that $X\otimes C\cong C$.
  \end{itemize}
\end{theorem}

Naturally, this implies

\begin{corollary}\label{cor:notens}
  The category $\cC^*_1$ of unital $C^*$-algebras is not tensored over $\cat{CMet}$.  \qedhere
\end{corollary}

Some auxiliary notation will be useful in handling (possible) tensors over \cat{CMet}, both in $\cC^*_1$ and, later, in $\cC^*_{c,1}$. Note that the {\it \cat{Set}}-valued version of \Cref{eq:hvhc} is always representable, with $\cC$ either $\cC^*_1$ or $\cC^*_{c,1}$. In both cases the representing object, which we denote by
\begin{equation*}
  X\otimes_{\cat{Set}}C\in \cC = \cC^*_1\text{ or }\cC^*_{c,1},
\end{equation*}
can be constructed as follows:
\begin{itemize}
\item first form the coproduct of copies $C_x$ of $C$ indexed by $x\in X$;
\item and then impose the additional constraints
  \begin{equation*}
    \|a_x-a_{x'}\|\le d(x,x'),\ \forall a\in C_1:=\text{unit ball of }C
  \end{equation*}
  where $a_x\in C_x$ denotes the copy of $a\in C$.
\end{itemize}
We thus always have a canonical identification
\begin{equation}\label{eq:justsetsfunctor}
  \mathrm{hom}_{\cC}(X\otimes_{\cat{Set}} C,-)
  \cong
  \mathrm{hom}_{\cV}(X,\mathrm{hom}_{\cC}(C,-))
\end{equation}
of \cat{Set}-valued functors. Both sides are metric spaces, and the metric tensor product $X\otimes C\in \cat{CMet}$ will exist precisely when that canonical morphism is isometric. In that case, of course, we will have
\begin{equation}\label{eq:xcxc}
  X\otimes C\cong X\otimes_{\cat{Set}}C. 
\end{equation}
We take this discussion for granted in the sequel.

\pf{th:notensfull}
\begin{th:notensfull}
  That the tensor products are as described in the three listed cases is an easy check, so we focus on proving the converse: that as soon as
  \begin{equation*}
    \dim C\ge 2
    \quad\text{and}\quad
    |X|\ge 2
  \end{equation*}
  the metric tensor product \Cref{eq:xcxc} does not exist. To that end we will construct
  
  \begin{itemize}
  \item a net
    \begin{equation*}
      (f_{\alpha})_{\alpha} = (f_{x,\alpha},\ x\in X)_{\alpha}
    \end{equation*}
    of $X$-tuples of morphisms
    \begin{equation*}
      f_{x,\alpha}:C\to A
    \end{equation*}
    such that
    \begin{equation*}
      \|f_{x,\alpha}-f_{x',\alpha}\|_{\infty}\le d(x,x')\quad\text{on the unit ball }C_1\subset C;
    \end{equation*}    
  \item with each $(f_{x,\alpha})_{\alpha}$ converging, respectively, to some $f_x:C\to A$ uniformly in $x$ and on the unit ball $C_1$;
  \item but such that the morphisms 
    \begin{equation*}
      f_{\alpha}:X\otimes_{\cat{Set}}C\to A
    \end{equation*}
    respectively induced by $f_{x,\alpha}$ do {\it not} converge to the corresponding morphism
    \begin{equation*}
      f:X\otimes_{\cat{Set}}C\to A
    \end{equation*}
    induced by $f_x$ uniformly on the unit ball of $X\otimes_{\cat{Set}}C$.
  \end{itemize}
  This will then show that the identification \Cref{eq:justsetsfunctor} (evaluated at $A$) is only one of sets, and not a homeomorphism.

  We start with the $f_x:C\to A$ (which will later be the limits, respectively, of $(f_{x,\alpha})_{\alpha}$); they will all be equal to a fixed (unital) embedding
  \begin{equation*}
    \iota:C\to A:=B(\cH) = \text{bounded operators on a Hilbert space $\cH$};
  \end{equation*}
  one always exists, by the {\it Gelfand-Naimark theorem} (\cite[Corollari II.6.4.10]{blk}). We will typically suppress $\iota$ and identify $C$ with its realization inside $A=B(\cH)$.

  Because $C$ is assumed at least 2-dimensional, there must be a unitary $u\in C$ that fails to commute with some unitary $v_{\alpha}\in B(\cH)$. Furthermore, we can choose $v_{\alpha}$ to be arbitrarily close to 1, so that
  \begin{equation*}
    \|v_{\alpha}-1\|\xrightarrow[\quad\alpha\quad]{} 0
  \end{equation*}
  (the $\alpha$s are the elements of an otherwise unspecified directed poset).

  Fix distinct elements $x_{0,1}\in X$ (assumed to exist: $|X|\ge 2$). If $v_{\alpha}$ is sufficiently close to 1 then it is connectable to 1 by a short path of unitaries (\cite[Proposition 4.2.4 and its proof]{wo}), so in particular we can assume those paths are shorter than $\ell:=d(x_0,x_1)$. We can now
  \begin{itemize}
  \item map $\{x_0,x_1\}$ contractively to the endpoints of a segment $I_{\alpha}$ of length $\|v_{\alpha}-1\|$;
  \item extend that map to a contraction from $X$ to $I_{\alpha}$, given that the latter is an {\it injective} object in the category \cat{Met} \cite[Theorem 4.7]{kk};
  \item and further map $I_{\alpha}$ contractively onto a path $\gamma_{\alpha}$ of unitaries connecting $1$ and $v_{\alpha}$, so as to obtain a contraction
    \begin{equation}\label{eq:phia}
      \varphi_{\alpha}:X\to \gamma_{\alpha}\subset B(\cH),\quad \varphi_{\alpha}(x_0)=1,\ \varphi_{\alpha}(x_1)=v_{\alpha}. 
    \end{equation}
  \end{itemize}
  Finally, define
  \begin{equation*}
    f_{x,\alpha}:C\to B(\cH),\ f_{x,\alpha}:=\varphi_{\alpha}(x)\ \cdot\ \varphi_{\alpha}(x)^*.
  \end{equation*}
  In words, this is conjugation by the unitary $\varphi_{\alpha}(x)$; rescaling the metrics involved slightly if necessary we can assume that \Cref{eq:phia} were in fact $C$-contractive for small $C>0$, so that
  \begin{equation*}
    X\ni x\mapsto f_{x,\alpha}\in \mathrm{hom}_{\cC^*_1}(C,B(\cH))
  \end{equation*}
  is contractive for each $\alpha$ (the right-hand side being metrized as usual, uniformly on the unit ball of $C$).

  The $\varphi_{\alpha}$ take values close to 1 for large $\alpha$ by construction, so the $X$-uniform convergence
  \begin{equation*}
    (f_{x,\alpha})_{\alpha}\to f_x = \iota:C\to B(\cH)
  \end{equation*}
  follows. On the other hand though, consider the unitaries
  \begin{equation*}
    w_{\alpha} :=
    v_{\alpha}uv_{\alpha}^* u^* =
    f_{\alpha}(u_{x_1}u_{x_0}^*)
    \in f_{\alpha}(X\otimes_{\cat{Set}}C).
  \end{equation*}
  They are non-scalar and converge to 1 in norm, so by the norm-continuity of the spectrum for normal operators (e.g. \cite[Problem 105]{hlm-hs}) their spectra will be contained in a small but non-trivial arc around $1\in \bS^1\subset \bC$. It follows that no matter how large $\alpha_0$ and $n_0$ are, we can find
  \begin{equation*}
    \alpha\ge \alpha_0,\ n\ge n_0,\ w_{\alpha}^n\text{ uniformly far from }1. 
  \end{equation*}
  In other words, we cannot have convergence
  \begin{equation*}
    w_{\alpha}^n = f_{\alpha}\left((u_{x_1}u_{x_0}^*)^n\right)\longrightarrow f\left((u_{x_1}u_{x_0}^*)^n\right) = 1
  \end{equation*}
  uniformly in $n$; or again:
  \begin{equation*}
    f_{\alpha} \centernot\longrightarrow f
  \end{equation*}
  uniformly on the unit ball of $X\otimes_{\cat{Set}}C$.
\end{th:notensfull}

Contrast \Cref{th:notensfull} with its commutative version: \cite[Proposition 3.11]{ck} states that unlike $\cC^*_1$, the category of commutative (unital) $C^*$-algebras {\it is} tensored over \cat{CMet}.

\subsection{Finite presentability}\label{subse:fpres}

\cite[Proposition 5.19]{ar-ap} classifies those finite metric spaces that are (enriched) $\aleph_0$-generated in \cat{CMet} with respect to the class of isometries: they are precisely the discrete ones (i.e. those with infinite pairwise distances; see \Cref{se:prel}). As we will soon see, {\it assuming} finiteness is not necessary:

\begin{theorem}\label{th:fgen}
  The objects in \cat{CMet} isometry-$\aleph_0$-generated in the enriched sense are precisely the finite discrete metric spaces.
\end{theorem}

A number of preliminary observations will simplify the main line of attack. The following terminology will be useful.

\begin{definition}
  The {\it finite-metric components} (or just plain `components', when context permits it) of a metric space are the maximal subspaces on which the metric takes finite value.
\end{definition}

Clearly, an object of \cat{CMet} is the coproduct of its finite-distance components (\cite[Exercise 1.1.2 and discussion preceding it]{bbi}); the same goes for \cat{Met}.

\begin{lemma}\label{le:fincomp}
  An isometry-$\aleph_0$-generated object in \cat{CMet} has finitely many components.
\end{lemma}
\begin{proof}
  Every space $(X,d)$ is the directed union of its subspaces $X_{\cF}$, unions of finite families $\cF$ of $X$-components. $\aleph_0$-generation then requires that $X$ be equal to one of the $X_{\cF}$, for otherwise the identity $X\to X$ would not be approximable by a morphism $X\to X_{\cF}$.
\end{proof}

\Cref{le:fincomp} reduces the problem to finite coproducts. Note next that it is enough to consider the individual components themselves.

\begin{lemma}\label{le:compenough}
  Let $X_i\in \cat{CMet}$, $1\le i\le n$ be a finite family of objects. The coproduct
  \begin{equation*}
    X:=\coprod_{i=1}^n X_i
  \end{equation*}
  is isometry-$\aleph_0$-generated in the enriched sense if and only if the $X_i$ are.
\end{lemma}
\begin{proof}
  This is almost entirely formal. If
  \begin{equation*}
    Y=\varinjlim_{\alpha} Y_{\alpha}
  \end{equation*}
  is a directed colimit of isometries, then the canonical morphism
  \begin{equation*}
    \varinjlim_{\alpha} \mathrm{hom}(X,Y_{\alpha})\to \mathrm{hom}(X,Y)
    \quad\text{in}\quad
    \cat{CMet}
  \end{equation*}
  is
  \begin{align*}
    \varinjlim_{\alpha} \mathrm{hom}(X,Y_{\alpha}) &\cong \varinjlim_{\alpha}\prod_{i=1}^n \mathrm{hom}(X_i,Y_{\alpha})\\
                                                   &\cong \prod_{i=1}^n \varinjlim_{\alpha} \mathrm{hom}(X_i,Y_{\alpha})
                                                     \quad{(\text{see below})}\\
                                                   &\to \prod_{i=1}^n \mathrm{hom}(X_i,Y)\\
                                                   &\cong \mathrm{hom}(X,Y).
  \end{align*}
  The second isomorphism is the commutation of finite products and directed colimits of isometries in \cat{CMet}, which is vary similar to the analogous statement in the category of sets (\cite[\S IX.2, Theorem 1]{mcl}), and admits a parallel proof.

  This is an isomorphism if and only if the individual components 
  \begin{equation*}
    \varinjlim_{\alpha} \mathrm{hom}(X_i,Y_{\alpha})\to \mathrm{hom}(X_i,Y)
  \end{equation*}
  are, hence the conclusion.
\end{proof}

Consequently:

\begin{corollary}\label{cor:fdistenough}
  The isometry-$\aleph_0$-generated objects in $\cat{CMet}$ are those that
  \begin{itemize}
  \item have finitely many finite-distance components;
  \item all of which are themselves isometry-$\aleph_0$-generated.
  \end{itemize}
\end{corollary}
\begin{proof}
  Immediate from \Cref{le:fincomp,le:compenough}. 
\end{proof}

\pf{th:fgen}
\begin{th:fgen}
  \Cref{cor:fdistenough} reduces the problem to showing that a non-empty finite-distance $\aleph_0$-generated object $(X,d)$ (fixed throughout the sequel) must be a singleton. The proof proceeds in a number of stages.

  \begin{enumerate}[(1)]
  \item\label{item:20} {\bf $X$ is compact.} Given that it is complete by assumption, this will follow from \cite[Theorem 45.1]{mnk} as soon as we show that $X$ is {\it totally bounded} \cite[\S 45, Definition preceding Example 1]{mnk}: for every $\varepsilon$, $X$ can be covered with finitely many $\varepsilon$-radius balls. To see this, note that $X$ is the directed union of its finite subspaces (equipped with the restricted metric)
    \begin{equation*}
      (X,d) = \varinjlim (F,d),\ F\subseteq X\ \text{finite}. 
    \end{equation*}
    Finite generation in the present (enriched) context means that for every $\varepsilon>0$ the identity $X\to X$ is uniformly $\varepsilon$-approximable by a contraction $X\to F$, whence every point is within $\varepsilon$ of one of the finitely many elements of $F$.
    
  \item {\bf $X$ is a singleton.} Suppose not. We can then fix $\varepsilon>0$ sufficiently small that (by finite generation) the identity $X\to X$ is uniformly $\varepsilon$-close to a contraction
    \begin{equation*}
      X\xrightarrow[]{\quad\pi\quad} F\subseteq X
    \end{equation*}
    factoring through some non-singleton finite $F\subseteq X$. The fibers
    \begin{equation*}
      X_p:=\pi^{-1}(p),\ p\in F
    \end{equation*}
    are {\it clopen} (i.e. closed and open) and partition $X$; denote by $\ell>0$ the smallest distance between two of them, say $X_{p}$ and $X_q$. By compactness that distance is actually achieved:
    \begin{equation*}
      \ell=d(X_p,X_q):=\inf\{d(x,y)\ |\ x\in X_p,\ y\in X_q\}
    \end{equation*}
    is in fact a minimum:
    \begin{equation*}
      d(x,y)=\ell\quad\text{for some}\quad x\in X_p,\ y\in X_q. 
    \end{equation*}
    The embedding $\iota:{\bf 2}_{\ell}\to X$ sending the two points to $x$ and $y$ now admits a contractive retraction (i.e. left inverse) $r:X\to {\bf 2}_{\ell}$ sending $X_p$ to one of the points and everything else to the other point.

    That the $\aleph_0$-generation property survives under retractions is a simple exercise, so ${\bf 2}_{\ell}$ must be isometry-$\aleph_0$-generated. This, though, contradicts \cite[Proposition 5.19]{ar-ap}.
  \end{enumerate}
  This finishes the proof.
\end{th:fgen}

The precise analogue of \Cref{th:fgen} holds in \cat{CPMet}:

\begin{theorem}\label{th:fgenp}
  The objects in \cat{CPMet} isometry-$\aleph_0$-generated in the enriched sense are precisely the finite discrete metric spaces.
\end{theorem}
\begin{proof}
  The \cat{CPMet} version of \Cref{cor:fdistenough} goes through just as easily, so we are again reduced to showing that a non-empty finite-distance $(X,d_X)\in \cat{CPMet}$ is isometry-$\aleph_0$-generated only if it is a singleton (the `if' implication being obvious).

  The proof strategy for this last claim will be very different, as most devices employed in the proof of \Cref{th:fgen} are absent here (we cannot work with finite spaces, etc., since every metric in sight must be intrinsic). We instead proceed to construct a non-approximable morphism
  \begin{equation}\label{eq:iotaxys}
    \iota:(X,d_X)\to (Y,d:=d_Y)\cong \varinjlim (Y_n,d_Y|_{Y_n})
  \end{equation}
  for (isometrically embedded) $Y_n\subset (Y,d)$ as follows.
  \begin{itemize}
  \item First, attach a metric segment $\gamma_x$ of length 1 to every point $x\in X$. This attachment occurs only at a single endpoint of $\gamma_x$ (which then becomes identified with $x$); the other endpoint is, say, $p_x$.

    Denote this (intermediate) space by $(Z,d_Z)$. It is complete by \Cref{th:tree}, as in the proof of \Cref{pr:cvxcmpl}. It is also of course a path metric space, being a gluing of such \cite[discussion following Exercise 3.1.13]{bbi}.
  \item Next, connect any two newly-added points
    \begin{equation*}
      p\in \gamma_x,\ q\in \gamma_y,\ x\ne y\in X
    \end{equation*}
    on {\it distinct} segments $\gamma_x$ and $\gamma_y$ with a metric segment $\gamma_{p,q}$ of length $d_Z(p,q)$. The result (once more a complete path metric space) will be our $(Y,d_Y)$.
  \item This all falls within the scope of \Cref{th:largepush} with $C=1$: first gluing along points to produce $Z$, and then gluing along isometrically embedded two-point spaces. It follows from the selfsame \Cref{th:largepush} that the various component spaces (the original $(X,d)$, the segments $\gamma_x$ and the $\gamma_{p,q}$) all embed isometrically into the end result $(Y,d_Y)$.

    We henceforth refer to the single ambient distance $d_Y$ as $d$.
  \item As just noted, the initial space $(X,d_X)$ embeds isometrically into $(Y,d)$; that identification is the map \Cref{eq:iotaxys}.
  \item It remains to identify the $Y_n\subseteq Y$, which we index by positive integers $n\in \bZ_{>0}$. By definition, $Y_n$ consists of
    \begin{itemize}
    \item the points on the partial segments
      \begin{equation*}
        \gamma_{n,x}:=\left\{p\in \gamma_x\ |\ d(p,X) = d(p,x)\ge \frac 1n\right\},\ x\in X;
      \end{equation*}
    \item and the points on those $\gamma_{p,q}$ that connect these:
      \begin{equation*}
        \gamma_{p,q}
        \quad\text{for}\quad
        p\in \gamma_{n,x},\ q\in \gamma_{n,y},\ x\ne y. 
      \end{equation*}
    \end{itemize}
  \end{itemize}
  By construction, the $Y_n$ are path metric spaces: points on the same $\gamma_{n,x}$ are already on a metric segment ($\gamma_{n,x}$ itself), while those on distinct $\gamma_{n,x_{0,1}}$, $x_0\ne x_1$ are connected by their own dedicated metric segment (one of the $\gamma_{p,q}$). It is also clear that $Y$ is the directed colimit (in \cat{Met}, or \cat{CMet}, or \cat{CPMet}) of the $Y_n$, as the near endpoints of the $\gamma_{n,x}$ respectively approach $x\in X$ uniformly in $n$.

  It remains to argue that if $X$ has at least two points, then $\iota:X\to Y$ is not arbitrarily approximable in the path-metric sense by morphisms into the $Y_n$.

  Consider two points $x\ne x'\in X$. Since the latter is intrinsic, there is a path
  \begin{equation*}
    \gamma:\left(I:=[0,1]\right)\to X
  \end{equation*}

  connecting $x$ and $x'$ (not necessarily a contraction). We will show that there cannot be, for arbitrarily small $\varepsilon$, paths $\gamma_n:I\to Y_n$ connectable to $\gamma$ by a path in
  \begin{equation*}
    \cat{Cont}(I\to Y):=\text{continuous maps $I\to Y$}
  \end{equation*}
  of length $<\varepsilon$. Indeed, if $\gamma_n:I\to Y_n$ is uniformly close to $\gamma$ then the endpoints of $\gamma_n$ must lie on $\gamma_{n,p}$ and $\gamma_{n,q}$ respectively, for small
  \begin{equation*}
    d(x,p)\text{ and }d(y,q). 
  \end{equation*}
  In particular, this imposes a uniform positive lower bound on the distance $d(p,q)$, and hence a uniform lower bound $\ell>0$ on the lengths of the segments $\gamma_{p,q}$ that such a $\gamma_n$ must traverse. But then any path connecting, say, the midpoint of $\gamma_{p,q}$ with a point on $X\in Y$ must have length at least $\frac{\ell}2$, and hence cannot be arbitrarily small.

  This concludes the proof.
\end{proof}

The following consequence of (the proof of) \Cref{th:fgenp} answers (negatively) the question of whether metric segments are isometry-$\aleph_0$-generated, asked in \cite[Remark 6.9]{dr}.

\begin{corollary}\label{cor:fgenc}
  The objects in \cat{CCMet} isometry-$\aleph_0$-generated in the enriched sense are precisely the finite discrete metric spaces.
\end{corollary}
\begin{proof}
  If $(X,d)$ is convex then the spaces $Y$ and $Y_n$ constructed in the proof of \Cref{th:fgenp} are also convex, for instance by \Cref{pr:cvxcmpl}.
\end{proof}



\addcontentsline{toc}{section}{References}

\Addresses

\end{document}